\documentclass{article}
\usepackage[utf8]{inputenc}
\usepackage{fullpage}
\usepackage{amssymb}
\usepackage{amsmath}
\usepackage{amsthm}
\usepackage{enumerate}
\usepackage{CJK}
\usepackage{color}
\usepackage{mathrsfs}
\usepackage{array}
\newtheorem{theorem}{Theorem}[section]
\newtheorem{proposition}[theorem]{Proposition}
\newtheorem{definition}[theorem]{Definition}
\newtheorem{corollary}[theorem]{Corollary}
\newtheorem{lemma}[theorem]{Lemma}
\newtheorem{remark}[theorem]{Remark}
\numberwithin{equation}{section} \numberwithin{theorem}{section}
\numberwithin{equation}{section}

\newtheorem{Def}[theorem]{Definition}
\newtheorem{example}[theorem]{Example}
\numberwithin{equation}{section}

\def\N{I\!\!N}

\newcommand{\ve}{\eta}

\def\Limsup{\mathop{{\rm Lim}\,{\rm sup}}}
\def\Liminf{\mathop{{\rm Lim}\,{\rm inf}}}
\def\ve{\varepsilon}
\def\tilde{\widetilde}
\def\emp{\emptyset}

\def\ox{\overline{x}}
\def\oy{\overline{y}}

\def\tto{\rightrightarrows}

\def\Tilde{\widetilde}
\def\Bar{\overline}

\def\ve{\varepsilon}
\def\epsilon{\varepsilon}
\def\ox{\bar{x}}
\def\oy{\bar{y}}

\def\gph{\mbox{\rm gph}\,}

\def\dim{\mbox{\rm dim}\,}

\def\dn{\downarrow}
\def\O{\Omega}
\def\ph{\varphi}
\def\emp{\emptyset}
\def\oR{\Bar{\R}}
\def\lm{\lambda}

\def\dd{\delta}

\def\Th{\Theta}
\def \N{{\rm I\!N}}
\def \R{{\rm I\!R}}

\def\Limsup{\mathop{{\rm Lim}\,{\rm sup}}}

\def\Limsup{\mathop{{\rm Lim}\,{\rm sup}}}

\numberwithin{equation}{section}

\usepackage{cancel}
\usepackage{ulem}
\usepackage{soul}

\title{Stability Criteria and Calculus Rules via Conic Contingent Coderivatives in Banach Spaces}

\author{Boris S. Mordukhovich\thanks{Department of Mathematics, Wayne State University
Detroit MI 48202 (e-mail: aa1086@wayne.edu). Research of this author was partly supported by the U.S. National Science Foundation under grants DMS-1808978 and DMS-2204519, by the Australian Research Council under grant DP-190100555, and by Project 111 of China under grant D21024.}
\and Pengcheng Wu\thanks{Department of Applied Mathematics, The Hong Kong Polytechnic University, Hong Kong
(e-mail: pcwu0725@163.com).}
\and Xiaoqi Yang\thanks{Department of Applied Mathematics, The Hong Kong Polytechnic University, Kowloon, Hong Kong (mayangxq@polyu.edu.hk). Research of this author was partly supported by the Research Grants Council of Hong Kong (PolyU  15205223).}}
\begin{document}
\maketitle

{\bf Abstract.} This paper addresses the study of novel constructions of variational analysis and generalized differentiation that are appropriate for characterizing robust stability properties of constrained set-valued mappings/multifunctions between Banach spaces important in optimization theory and its applications. Our tools of generalized differentiation revolves around the newly introduced concept of $\varepsilon$-regular normal cone to sets and associated coderivative notions for set-valued mappings. Based on these constructions, we establish several characterizations of the central stability notion known as the relative Lipschitz-like property of set-valued mappings in infinite dimensions. Applying a new version of the constrained extremal principle of variational analysis, we develop comprehensive sum and chain rules for our major constructions of conic contingent coderivatives for multifunctions between appropriate classes of Banach spaces.\\[0.5ex]
\vspace*{-0.1in}

\section{Introduction}\label{sec:intro}

This paper is devoted to developing adequate machinery of variational analysis and generalized differentiation to study {\it robust well-posedness} properties of set-valued mappings {\it subject to constraints}. The importance of such properties and their applications to optimization, control, and related areas has been highly recognized over the years, while the main attention was paid to unconstrained mappings in finite and infinite dimensions; see, e.g., the books \cite{Bonnans2000,Borwein2005,Ioffe,Mordukhovich2006,m18,Rockafellar1998,Thibault} with the extensive bibliographies therein. The study and applications of such properties for set-valued mappings {\it relative} to given sets has begun quite recently in \cite{Yang2021,Mordukhovich2023,Yao2023}, where the challenges in the study of the constrained properties in comparison with their unconstrained counterparts have been well realized. 

In this paper, we focus on investigating the {\it relative Lipschitz-like property} of set-valued mappings between Banach spaces, which in fact is equivalent to the {\it relative metric regularity} and {\it relative linear openness} of the inverses; see \cite{Mordukhovich2023}. We address here infinite-dimensional settings of Banach spaces, although the major results obtained below are new even in finite dimensions. Besides mathematical reasons, the main motivation for our research comes from the fact that infinite-dimensional problems naturally arise in many areas of applied mathematics and practical modeling such as dynamic optimization and equilibria, stochastic programming, finance, systems control, economics etc.

To accomplish our goals, we introduce a novel type of normals to sets labeled as the {\it regular $\varepsilon$-normal cone}. It is essentially different from the ``collection of $\varepsilon$-normals" (which is not a cone) overwhelmingly used in infinite-dimensional variational analysis; see \cite{Mordukhovich2006}. Based on this new normal cone and invoking also the {\it contingent cone} to constraint sets, we define various coderivative constructions for set-valued mappings and show that these newly introduced coderivatives lead us to {\it complete characterizations} (of the neighborhood and pointbased types) of the relative Lipschitz-like property of constrained multifunctions with the precise calculations of the {\it exact Lipschitzian bound}.  

The efficient usage of the established stability criteria in the case of structured multifunctions, which appear in constrained problems and applications, requires well-developed {\it calculus rules}. We derive here comprehensive {\it pointbased chain} and {\it sum rules} for {\it conic contingent coderivatives}, the limiting constructions from the obtained characterizations of the relative Lipschitz-like property for general multifunctions. Our proofs are mainly based on the fundamental {\it extremal principle} of variational analysis and its new {\it relative} version needed for the study of constrained set systems.\vspace*{0.02in}

The rest of the paper is structured as follows. In Section~\ref{sec:pre}, we introduce the $\varepsilon$-regular normal cone and recall the needed preliminaries from variational analysis and generalized differentiation. Sections~\ref{sec:chara} and \ref{sec:point} are devoted,  respectively, to establishing neighborhood and pointbased necessary conditions, sufficient conditions, and characterizations of the relative Lipschitz-like property. In Sections~\ref{sec:chain} and \ref{sec:sum}, we derive pointbased chain and sum rules, respectively, for the newly introduced conic contingent coderivatives with establishing the novel relative extremal principle and relative fuzzy intersection rule for set systems in Banach spaces. The concluding Section~\ref{sec:conc} summarizes the main achievements of the paper and discusses some directions of the future research. \vspace*{-0.1in}

\section{ $\varepsilon$-Regular Normal Cone and Preliminaries}\label{sec:pre}

First we recall some standard {\it notations} used in this paper. Given a Banach space $X$ (which is assumed throughout below unless otherwise stated) and its topological dual $X^*$, the norms of $X$ and $X^*$ are denoted by $\|\cdot\|_X$ and $\|\cdot\|_{X^*}$, respectively. Denote the closed unit balls of $X$ and $X^*$ by $B_X$ and $B_{X^*}$, respectively, and use
$\mathbb S:=\{x\in X~|~\|x\|_X=1\}$ for the unit sphere. For $x\in X$ and $x^*\in X^*$, the symbol $\langle x^*,x\rangle$ indicates the canonical pairing between $X$ and $X^*$. We use the notation $``x_k\to x"$ for the strong convergence of sequences, while $``x_k\rightharpoonup x"$ denotes the weak convergence, i.e., 
$\langle x^*,x_k\rangle\to\langle x^*,x\rangle$ for all $x^*\in X^*$. The notation $``x_k^*\rightharpoonup^* x^*"$ stands for the weak$^*$ convergence, i.e., $\langle x_k^*,x\rangle\to\langle x^*,x\rangle$ for any $x\in X$. The symbol $``x_k\stackrel{\Omega}{\longrightarrow}x"$ tells us that $x_k\to x$ with $x_k\in \Omega$ as $k\to\infty$. The collection of natural numbers is denoted by $\N:=\{1,2,\ldots\}$.

Recall that a multifunction $F\colon X\tto Y$ between Banach spaces is {\it positively homogeneous} if $0\in F(0)$ and $F(\lm x)=\lm F(x)$ for all $x\in X$ and $\lm>0$. The {\it norm} of a positively homogeneous multifunction is defined by
\begin{equation}\label{norm}
\|F\|:=\sup\big\{\|y\|\;\big|\;y\in F(x),\;\|x\|\le 1\big\}.
\end{equation}

Given a set $\Omega\subset X$ with $\bar x\in\Omega$, the symbol $T(\bar x;\Omega)$ stands for the {\it tangent/contingent cone} to $\Omega$ at $\bar x$. That is, $w\in T(\bar x;\Omega)$ if there exist sequences $t_k\downarrow 0$ and $\{w_k\}\subset X$ such that $w_k\to w$ as $k\to\infty$ and $\bar x+t_kw_k\in\Omega$ whenever $k\in\N$. Fixing any $\varepsilon\ge 0$, we now introduce our major new notion of the {\it $\varepsilon$-regular normal cone} to $\Omega$ at $\bar x\in\Omega$ by
\begin{equation}\label{e-cone}
\widehat N_\varepsilon^c(\bar x;\Omega):=\Big\{x^*\in X^*~\Big|~\limsup\limits_{x\stackrel{\Omega}{\longrightarrow}\bar x}\frac{\langle x^*,x-\bar x\rangle}{\|x-\bar x\|_X}\leq\varepsilon\|x^*\|_{X^*}\Big\},
\end{equation}
where ``$c$" stands for ``cone". Observe that $\widehat N_\varepsilon^c(\bar x;\Omega)$ is indeed a cone for any $\varepsilon\ge 0$, in contrast to the {\it $\varepsilon$-normal set} to $\Omega$ at $\bar x$ defined by
\begin{equation}\label{e-set}
\widehat N_\varepsilon(\bar x;\Omega):=\Big\{x^*\in X^*~\Big|~\limsup\limits_{x\stackrel{\Omega}{\longrightarrow}\bar x}\frac{\langle x^*,x-\bar x\rangle}{\|x-\bar x\|_X}\leq\varepsilon\Big\},
\end{equation}
which is broadly used in variational analysis; see \cite{Mordukhovich2006}. When $\varepsilon=0$, both sets in \eqref{e-cone} and \eqref{e-set} reduce to the {\it regular normal cone} to $\Omega$ at $\bar x$ denoted by $\widehat N(\bar x;\Omega)$. It is well known that the latter notion is associated with the {\it regular subdifferential} of an extended-real-valued function $\ph\colon X\to\oR:=(-\infty,\infty]$ at $\ox$ with $\ph(\ox)<\infty$ given by
\begin{equation}\label{sub}
\widehat\partial\varphi(\bar x):=\Big\{x^*\in X^*~\Big|~\limsup\limits_{x\to\bar x}\frac{\varphi(x)-\varphi(\bar x)-\langle x^*,x-\bar x\rangle}{\|x-\bar x\|_X}\ge 0\Big\}.
\end{equation}
Consider further a multifunction $S:X\rightrightarrows Y$ between Banach spaces with the graph
\begin{equation*}
\gph S:=\big\{(x,y)\in X\times Y\;\big|\;y\in S(x)\big\}
\end{equation*}
and denote by ${\cal N}(z)$ the collection of (open) neighborhoods around a point $z$.

\begin{definition}\label{lip-like}
A mapping $S:X\rightrightarrows Y$ has the {\sc Lipschitz-like property relative to} $\Omega$ around $(\bar x,\bar y)\in\gph S$ with $\bar x\in\Omega$ if 
there are neighborhoods $V\in\mathcal{N}(\bar x)$ and $W\in\mathcal{N}(\bar y)$ together with a constant $\kappa\geq0$ such that
$$
S(x')\cap W\subset S(x)+\kappa\|x'-x\|_X B_X\;\mbox{ for all }\;x,x'\in \Omega \cap V.
$$
The {\sc exact Lipschitzian bound} of $S$ {\sc relative to} $\Omega$ at $(\bar x,\bar y)$ is defined by
$$
\begin{aligned}
\operatorname{lip}_\Omega S(\bar x,\bar y):=\inf\big\{\kappa\geq0&~\big|~\exists V\in\mathcal{N}(\bar x),W\in\mathcal{N}(\bar y)~\operatorname{such~that}\\
&S(x')\cap W\subset S(x)+\kappa\|x'-x\|_X B_X\;\mbox{ for all }\;x,x'\in \Omega\cap V\big\}.
\end{aligned}
$$
\end{definition}

The {\it duality mapping} $J\colon X\tto X^*$ between a Banach space $X$ and its dual is given by
\begin{equation}\label{dual map}
J(x):=\big\{x^*\in X^*~\big|~\langle x^*,x\rangle=\|x\|_X^2,~\|x^*\|_{X^*}=\|x\|_{X}\big\}\;\mbox{ for all }\;x\in X.
\end{equation}
We clearly have $J(\lambda x)=\lambda J(x)$ whenever $\lambda\in\mathbb R$ and get by \cite[Proposition~2.12]{Bonnans2000} that $J(x)\not=\emptyset$ for all $x\in X$. It is well known from geometric theory of Banach spaces (see, e.g., \cite{fabian}) that any reflexive Banach space $X$ admits a renorming $\|\cdot\|$ such that both $X$ and $X^*$ are locally uniformly convex and $\|\cdot\|$ is Fr\'echet differentiable at each nonzero point of $X$. Unless otherwise stated, in what follows we always consider {\it reflexive Banach spaces} with such a renorming. Recall from \cite[Proposition~8]{Browder1983} that the duality mapping \eqref{dual map} in reflexive spaces is single-valued and bicontinuous.\vspace*{0.03in}

\begin{proposition}\label{JT}
Let $X$ be a reflexive Banach space, let $\Omega\subset X$ be a closed and convex set, and let $x^*\in J(T(\bar x;\Omega))$. Then for any $\varepsilon>0$, there exists $\delta>0$ such that we have
$$
x^*\in J\big(T(x;\Omega)\big)+\varepsilon B_{X^*}\;\mbox{ whenever }\;x\in\Omega\cap(\bar x+\delta B_X).
$$
\end{proposition}
\begin{proof}
Suppose on the contrary that the claimed inclusion fails and then find a number $\varepsilon_0>0$ and a sequence $x_k\stackrel{\Omega}{\longrightarrow}\bar x$ as $k\to\infty$ such that
\begin{equation}\label{notin}
x^*\notin J\big(T(x_k;\Omega)\big)+\varepsilon_0 B_{X^*}\;\mbox{ for all }\;k\in\mathbb N.
\end{equation}
By the aforementioned properties of the duality mapping, we deduce from $J^{-1}(x^*)\in T(\bar x;\Omega)$ and the well-known relationship
$$
T(\bar x;\Omega)\subset s\mbox{-}\liminf\limits_{x\stackrel{\Omega}{\longrightarrow}\bar x}T(x;\Omega),
$$
valid for convex sets $\Omega$, that there exists a sequence of $z_k\in T(x_k;\Omega)$ such that $z_k\to J^{-1}(x^*)$ for all $k\in\mathbb N$. Letting $x_k^*:=J(z_k)$, we get that
$$
x_k^*\in J\big(T(x_k;\Omega)\big)~~\operatorname{and}~~x_k^*\to x^*\;\mbox{ as }\;k\to\infty,
$$
which contradicts \eqref{notin} and completes the proof of the proposition.
\end{proof}
\vspace*{-0.15in}

\section{Neighborhood Conditions for Relative Lipschitz-like Multifunctions}\label{sec:chara}

The goal of this section is to derive necessary conditions, sufficient conditions, and complete characterizations of the Lipschitz-like property for set-valued mappings $S\colon X\to Y$ relative to sets $\O$ in terms of the {\it $\ve$-regular normal cone} \eqref{e-cone} to the {\it graph} of the constrained mapping $S|_\O$ with the usage of the {\it contingent cone} $T_\O(x)$ to the constraint set $\O$. According to our general scheme (cf.\ \cite{Mordukhovich2006,Mordukhovich2023}), such a construction can be viewed as the {\it $\ve$-coderivative} of constrained mapping in question; see Definition~\ref{coderivatives}{(i). Note the conditions derived in this section in the case of set-valued between reflexive Banach spaces are {\it neighborhood}, i.e., they involve points nearby the reference one.\vspace*{0.02in} 

First we present a neighborhood {\it necessary} condition for the Lipschitz-like property of multifunctions relative to {\it arbitrary} constraint sets.

\begin{lemma}\label{Aubin-TX0} Let $S:X\rightrightarrows Y$ be a set-valued mapping between reflexive Banach spaces $X$ and $Y$, let $\Omega\subset X$ be a nonempty set, and let $(\bar x,\bar y)\in\operatorname{gph} S|_\Omega$. If $S$ has the Lipschitz-like property relative to $\Omega$ around $(\bar x,\bar y)$ with constant $\kappa\ge 0$, then for any $\varepsilon\in(0,\frac{1}{\sqrt{\kappa^2+1}})$ there exists $\delta>0$ such that we have the condition
\begin{equation}\label{TX0=0}
\widehat N_\varepsilon^c\big((x,y);\gph S|_\Omega\big)\cap[J\big(T(x;\Omega)\big)\times\{0\}]=\{(0,0)\}
\end{equation}
whenever $(x,y)\in\gph S|_\Omega\cap[(\bar x+\delta B_X)\times(\bar y+\delta B_Y)]$.
\end{lemma}
\begin{proof} Having the imposed relative Lipschitz-like property gives us $\delta>0$ with
\begin{equation}\label{(b)}
S(x)\cap(\bar y+2\delta B_Y)\subset S(x')+\kappa\|x-x'\|_X B_X\;\mbox{ for all }\;x,x'\in \Omega\cap(\bar x+2\delta B_X).
\end{equation}
Picking $\varepsilon\in(0,\frac{1}{\sqrt{\kappa^2+1}})$, fix $(x,y)\in\text{gph}~S|_\Omega\cap[(\bar x+\delta B_X)\times(\bar y+\delta B_Y)]$ and $x^*\in T(x;\Omega)\setminus\{0\}$. By definition of the contingent cone, there exist $t_k\dn 0$ and $x_k^*\to x^*$ as $k\to\infty$ such that
\begin{equation*}
x_k:=x+t_kx_k^*\in \Omega\;\mbox{ for all }\;k\in\N,
\end{equation*}
which implies the existence of $k_0\in\N$ ensuring the estimates
\begin{equation*}
t_k<\frac{\delta}{2\|x^*\|_{X^*}} ~~\text{and}~~\|x_k^*-x^*\|_{X^*}<\frac{\|x^*\|_{X^*}}{4}\;\mbox{for all }\;k>k_0.
\end{equation*}
This clearly leads us to the inequalities
$$
\begin{aligned}
\|x_k-\bar x\|_X&\leq\|x_k-x\|_X+\|x-\bar x\|_X<\frac{\delta}{2\|x^*\|_{X^*}}\|x_k^*\|_{X^*}+\delta\\
&\leq\frac{\delta}{2\|x^*\|_{X^*}}(\|x^*\|_{X^*}+\|x^*-x_k^*\|_{X^*})
+\delta<2\delta.
\end{aligned}
$$
It follows from \eqref{(b)} that for all natural numbers $k>k_0$ we get
$$
y\in S(x)\cap(\bar y+2\delta B_Y)\subset S(x_k)+\kappa\|x-x_k\|_X.
$$
Therefore, there exist sequences $y_k\in S(x_k)$ and $\kappa_k\dn\kappa$ such that
$$
\|y_k-y\|_Y\leq\kappa_k\|x_k-x\|_X,
$$
which implies in turn that 
$$
\|(x_k,y_k)-(x,y)\|_{X\times Y}=\sqrt{\|x_k-x\|_X^2+\|y_k-y\|_Y^2}\leq\sqrt{\kappa_k^2+1}\|x_k-x\|_X.
$$
Furthermore, the convergence $x_k^*\to x^*$ as $k\to\infty$ allows us to find $k_1>k_0$ such that
$$
\langle J(x^*),x_k^*\rangle\geq\frac{1}{2}\|x^*\|^2>0\;\mbox{ for all }\;k>k_1.
$$
Thus for all $k>k_1$ we have
$$
\begin{aligned}
&\frac{\langle (J(x^*),0),(x_k,y_k)-(x,y)\rangle}{\|x^*\|_{X^*}\|(x_k,y_k)-(x,y)\|_{X\times Y}}
\geq\frac{t_k\langle J(x^*),x_k^*\rangle}{\sqrt{\kappa_k^2+1}\|x^*\|_{X^*}\|x_k-x\|_X}\\
=&\frac{1}{\sqrt{\kappa_k^2+1}}\frac{\langle J(x^*),x_k^*\rangle}{\|x^*\|_{X^*}\|x_k^*\|_X}
\to\frac{1}{\sqrt{\kappa^2+1}}>\varepsilon~~\text{as}~k
\to\infty,
\end{aligned}
$$
which yields $(J(x^*),0)\notin \widehat N_\varepsilon^c\big((x,u);\operatorname{gph} S|_\Omega\big)$. This verifies \eqref{TX0=0}  due to the arbitrary choice of
$(x,y)\in\operatorname{gph}~S|_\Omega\cap[(\bar x+\delta B_X)\times(\bar y+\delta B_Y)]$ and $x^*\in T(x;\Omega)\setminus\{0\}$.
\end{proof}

The next lemma provides a refinement of \eqref{TX0=0}, which is valid in the same setting for $S$ and $\O$ as in Lemma~{\rm\ref{Aubin-TX0}}, while being useful to establish the neighborhood sufficient condition for the relative Lipschitz-like property in Lemma~\ref{Sufficiency} below.

\begin{lemma}\label{Necessity} Let $S$, $\O$, and $(\ox,\oy)$ be the same as in Lemma~{\rm\ref{Aubin-TX0}}. If there exist numbers $\varepsilon\in(0,1)$ and $\delta>0$ such that \eqref{TX0=0} holds, then whenever $\kappa>\sqrt{\frac{1}{\varepsilon^2}-1}$ we find $\gamma\in(0,\varepsilon)$ ensuring the fulfillment of the estimate
\begin{equation}\label{nece1}
\|x^*\|\leq\kappa \|y^*\|\;\mbox{ when }\;(x^*,y^*)\in \widehat N_\gamma^c\big((x,y);\gph S|_\Omega\big)\cap[J\big(T(x;\Omega)\big)\times Y^*]
\end{equation}
for all $(x,y)\in\operatorname{gph}S|_\Omega\cap[(\bar x+\delta B_X)\times(\bar y+\delta B_Y)]$.
\end{lemma}
\begin{proof}
Set $\gamma:=\varepsilon-\frac{1}{\sqrt{\kappa^2+1}}$ and verify that \eqref{nece1} holds. Suppose on the contrary that there exist $(x,y)\in\operatorname{gph} S|_\Omega\cap[(\bar x+\delta B_X)\times(\bar y+\delta B_Y)]$ and $(x^*,y^*)\in \widehat N_{\gamma}^c\big((x,y);\operatorname{gph} S|_\Omega\big)\cap[J\big(T(x;\Omega)\big)\times Y^*]$ such that
\begin{equation*}
\|x^*\|_{X^*}>\kappa\|y^*\|_{Y^*}.
\end{equation*}
Considering further the pairs
$$
(x_N^*,y_N^*):=\frac{\kappa^2}{\kappa^2+1}(x^*,y^*)~~~\text{and}~~~
(x_T^*,y_T^*):=(x^*,0)-(x_N^*,y_N^*),
$$
we easily check the relationships
\begin{equation}\label{d4-2}
\begin{aligned}
\|(x_T^*,y_T^*)\|_{X^*\times Y^*}=&\left\|(x^*,0)-\frac{\kappa^2}{\kappa^2+1}(x^*,y^*)\right\|_{X^*\times Y^*}\\
=&\sqrt{\frac{1}{(\kappa^2+1)^2}\|x^*\|_{X^*}^2+\frac{\kappa^4}{(\kappa^2+1)^2}\|y^*\|_{Y^*}^2}\\
\leq&\frac{1}{\sqrt{\kappa^2+1}}\|x^*\|_{X^*},
\end{aligned}
\end{equation}
\begin{equation}\label{d4-1}
\|(x_N^*,y_N^*)\|_{X^*\times Y^*}\leq\|x^*\|_{X^*}.
\end{equation}
Then we deduce from \eqref{d4-2} the inequalities
$$
\begin{aligned}
\langle (x_T^*,y_T^*),(x',y')-(x,y)\rangle\leq&\|(x_T^*,y_T^*)\|_{X^*\times Y^*}\|(x',y')-(x,y)\|_{X\times Y}\\
\leq&\frac{1}{\sqrt{\kappa^2+1}}\|x^*\|_{X^*}\|(x',y')-(x,y)\|_{X\times Y},
\end{aligned}
$$
which, being combined with \eqref{d4-1} and  $(x_N^*,y_N^*)\in \widehat N_{\gamma}^c\big((x,u);\operatorname{gph} S|_\Omega\big)$, lead us to
$$
\begin{aligned}
&\limsup\limits_{(x',y')\stackrel{\text{gph}~S|_\Omega}{\longrightarrow}(x,y)}
\frac{\langle(x^*,0),(x',y')-(x,y)\rangle}{\|(x',y')-(x,y)\|_{X\times Y}}\\
=&\limsup\limits_{(x',y')\stackrel{\operatorname{gph} S|_\Omega}{\longrightarrow}(x,y)}
\frac{\langle(x_N^*,y_N^*)+(x_T^*,y_T^*),(x',y')-(x,y)\rangle}{\|(x',y')-(x,y)\|_{X\times Y}}\\
\leq&\limsup\limits_{(x',y')\stackrel{\operatorname{gph} S|_\Omega}{\longrightarrow}(x,y)}
\frac{\langle(x_N^*,y_N^*),(x',y')-(x,y)\rangle}{\|(x',y')-(x,y)\|_{X\times Y}}
+\frac{1}{\sqrt{\kappa^2+1}}\|x^*\|_{X^*}\\
\leq&\Big(\gamma+\frac{1}{\sqrt{\kappa^2+1}}\Big)\|x^*\|_{X^*}=\varepsilon\|x^*\|_{X^*},
\end{aligned}
$$
where the last inequality follows from the definition of $\varepsilon$-normal cone. The latter yields $(x^*,0)\in \widehat N_{\varepsilon}^c\big((x,y);\operatorname{gph} S|_\Omega\big)$, which clearly contradicts \eqref{TX0=0} and hence verifies the claimed condition \eqref{nece1}.
\end{proof}

Having Lemma~\ref{Necessity} in hand and employing fundamental tools of variational analysis and generalized differentiation, we show next that the neighborhood condition \eqref{nece1} is actually {\it sufficient} for the relative Lipschitz-like property. Observe that, in contrast to the necessary conditions in Lemmas~\ref{Aubin-TX0} and \ref{Necessity}, we now require that the graph of $S$ is {\it locally closed}, and the set $\O$ is {\it closed} and {\it convex}.

\begin{lemma}\label{Sufficiency}
Let $S\colon X\tto Y$ be a set-valued mapping between reflexive Banach spaces, and let $\Omega\subset X$ be a closed and convex set. Fix $(\bar x,\bar y)\in\operatorname{gph} S|_\Omega$, and assume that $\operatorname{gph} S$ is locally closed around $(\bar x,\bar y)$. If there exist $\delta>0$, $\kappa\geq0$, and $\gamma\in(0,1)$ such that condition \eqref{nece1} holds for all $(x,y)\in\operatorname{gph}~S|_\Omega\cap[(\bar x+\delta B_X)\times(\bar y+\delta B_Y)]$, then $S$ has the Lipschitz-like property relative to $\Omega$ around $(\bar x,\bar y)$ with constant $\kappa$.
\end{lemma}
\begin{proof}
Assume without loss of generality that $\operatorname{gph}S|_\Omega$ is closed in $X\times Y$ and define the positive number $\sigma:=\delta/4(\kappa+1)$. Arguing by contradiction, suppose that $S$ does not have the Lipschitz-like property relative to $\Omega$ around $(\bar x,\bar y)$ with constant $\kappa$. Then there are $x',x''\in \Omega\cap(\bar x+\sigma B_X)$ with $x'\not=x''$ and $y''\in S(x'')\cap (\bar y+\sigma B_Y)$ such that
\begin{equation}\label{notaubin}
d(y'';S(x'))>\kappa\|x'-x''\|_X:=\beta
\end{equation}
via the distance function in $Y$. 
Clearly, we have $0<\beta\leq2\kappa\sigma$. Let us now show that condition \eqref{nece1} from Lemma~\ref{Necessity} fails in this case whenever $\gamma\in(0,1)$. Define the extended-real-valued function $\varphi:X\times Y\to\oR$ by
$$
\varphi(x,y)=\|x-x'\|_X+\dd_{\operatorname{gph} S|_\Omega}(x,y),\quad(x,y)\in X\times Y,
$$
via the indicator function of the set $\gph S_\Omega$ and observe that $\ph$ is lower semicontinuous (l.s.c.) by the closedness of \text{gph}~$S|_\Omega$), that $\inf \varphi$ is finite, and that
$$
\varphi(x'',y'')\leq\inf \varphi+\frac{\beta}{\kappa}.
$$
Equip the product space $X\times Y$ with the norm
\begin{equation}\label{rho}
p(x,y):=\kappa\varrho\|x\|_X+\|y\|_Y,\;\mbox{ where }\; \varrho:={\rm min}\Big\{\frac{\gamma}{4},\frac{1}{2k}\Big\},
\end{equation}
and apply the fundamental Ekeland variational principle (see, e.g., \cite[Theorem~2.26]{Mordukhovich2006}) to $\ph$ on $X\times Y$ with the norm \eqref{rho}. This gives us a pair $(\tilde{x},\tilde{y})\in X\times Y$ such that
\begin{align}
&p(\tilde{x}-x'',\tilde{y}-y'')\leq\beta, \label{align1}\\
&\varphi(\tilde{x},\tilde{y})\leq\varphi(x'',u''),\;\mbox{ and} \label{align2}\\
&\varphi(\tilde{x},\tilde{y})\leq\varphi(x,y)+\frac{1}{\kappa}p(x-\tilde{x},y-\tilde{y})\;\mbox{ for all }\;(x,y)\in X\times Y. \label{align3}
\end{align}
It immediately follows from \eqref{align1} that
\begin{equation}\label{<beta}
\|\tilde y-y''\|_Y\leq\beta.
\end{equation}
So we have $\tilde x\not= x'$, since otherwise 
$$
d(y'';S(x'))=d(y'';S(\tilde x))\leq \|\tilde y-y''\|_Y\leq\beta,
$$
which contradicts \eqref{notaubin}. Furthermore, using the continuity of the duality mapping \eqref{dual map} due to the reflexivity of $X$ yields the existence of $\nu>0$ such that
\begin{equation}\label{Jcontinuous}
\|J(x'-x)-J(x'-\tilde x)\|_{X^*}\leq\frac{\varrho}{4}\|x'-\tilde x\|_X\;\mbox{ for all }\;x\in \tilde x+\nu B_X.
\end{equation}
Put $\eta:={\rm min}\{\frac{\varrho}{4}\|x'-\tilde x\|_X,\frac{\delta}{4},\varrho,\gamma\}$ and deduce from \eqref{align3} and \cite[Proposition~3.3(ii)]{Mordukhovich2023} that there exist $(x_i,y_i)\in\operatorname{gph} S|_\Omega\cap[(\tilde x+\eta B_X)\times (\tilde y+\eta B_Y)]$, $i=1,2,3$, such that
\begin{equation*}\label{(0,0)in}
(0,0)\in\widehat\partial\psi_1(x_1,y_1)+\widehat\partial\psi_2(x_2,y_2)+\widehat N\big((x_3,y_3);\operatorname{gph} S|_\Omega\big)+\eta[B_{X^*}\times B_{Y^*}],
\end{equation*}
where both $\psi_1$ and $\psi_2$ are the Lipschitz continuous convex functions defined by
$$
\psi_1(x,y):=\|x-x'\|_X~~\mbox{and}~~ \psi_2(x,y):=\frac{1}{\kappa}p(x-\tilde{x},y-\tilde{y}).
$$
Since we clearly have $x_1\not=x'$ due to $\tilde x\not=x'$ and $\|x_1-\tilde x\|_X<\|x'-\tilde x\|_X$, it follows from \cite[Example~3.36]{Nam}) that
$$
\widehat\partial\psi_1(x_1,y_1)=\bigg(\frac{J(x_1-x')}{\|x_1-x'\|_X},0\bigg)~~\operatorname{and}~~
\widehat\partial\psi_2(x_2,y_2)\subset \varrho B_{X^*}\times\frac{1}{\kappa}B_{Y^*}.
$$
Thus there exists $(u_1,v_1)\in B_{X^*}\times B_{Y^*}$ such that
\begin{equation}\label{normalcone}
(x^*,y^*)\in\widehat N\big((x_3,y_3);\operatorname{gph} S|_\Omega\big)\;\mbox{ with }
\end{equation}
\begin{equation*}\label{xstar}
x^*:=\frac{J(x_1-x')}{\|x'-x_1\|_X}-(\varrho+\eta) u_1~~\mbox{and}~~ y^*:=\left(\frac{1}{\kappa}+\eta\right)v_1.
\end{equation*}
This gives us by taking \eqref{Jcontinuous} into account that
$$
\begin{aligned}
&\left\|\frac{J(x'-x_i)}{\|x'-x_i\|_X}-\frac{J(x'-\tilde x)}{\|x'-\tilde x\|_X}\right\|_{X^*}\\
\leq&\left\|\frac{J(x'-x_i)}{\|x'-x_i\|_X}-\frac{J(x'-x_i)}{\|x'-\tilde x\|_X}\right\|_{X^*}
+\left\|\frac{J(x'-x_i)}{\|x'-\tilde x\|_X}-\frac{J(x'-\tilde x)}{\|x'-\tilde x\|_X}\right\|_{X^*}\\
\leq&\frac{\|x_i-\tilde x\|_X}{\|x'-\tilde x\|_X}+\frac{\varrho}{4}\leq\frac{\varrho}{2}~~\mbox{for}~~i=1,3,
\end{aligned}
$$
which shows therefore that
$$
\left\|\frac{J(x'-x_1)}{\|x'-x_1\|_X}- \frac{J(x'-x_3)}{\|x'-x_3\|_X}\right\|_{X^*}\leq\varrho.
$$
Define $\hat x^*:=\frac{J(x'-x_3)}{\|x'-x_3\|_X}$ and $\hat y^*:=\left(\frac{1}{\kappa}-\eta\right)v_1$. Then we get 
\begin{equation}\label{>kappa}
\|\hat x^*\|_{X^*}>\kappa\|\hat y^*\|_{Y^*}
\end{equation}
and find $u_2\in B_{X^*}$ satisfying
\begin{equation}\label{xystar}
\hat x^*=x^*+(\varrho+\eta)u_1+\varrho u_2~~\operatorname{and}~~\hat y^*=y^*-2\eta v_1.
\end{equation}
Observe that $\hat x^*\in J(T(x_3;\Omega))$ by the convexity of $\Omega$ and the homogeneity of $J$. It follows from \eqref{align2} that
$(\tilde{x},\tilde{u})\in\text{gph}~S|_X$,
and so $\|\tilde{x}-x'\|\leq\|x''-x'\|$. Employing further the triangle inequality yields
\begin{equation*}\label{x-x}
\begin{aligned}
\|x_3-\bar x\|_X\leq&\|x_3-\tilde x\|_X+\|\tilde{x}-x'\|_X+\|x'-\bar x\|_X\\
\leq&\eta+\|x''-\bar x\|_X+2\|x'-\bar x\|_X\leq\frac{\delta}{4}+3\sigma<\delta,
\end{aligned}
\end{equation*}
which being combined with \eqref{<beta} tells us that
$$
\|y_3-\bar y\|_Y\leq\|y_3-\tilde y\|_Y+\|\tilde y-y''\|_Y+\|y''-\bar y\|_Y
\leq\eta+\beta+\sigma<\delta.
$$
This shows that $(x_3,y_3)\in\operatorname{gph}~S|_\Omega\cap[(\bar x+\delta B_X)\times(\bar y+\delta B_Y)]$ and thus verifies \eqref{nece1} for $(x,y):=(x_3,y_3)$. Furthermore, it follows from \eqref{normalcone} and \eqref{xystar} that
$$
\begin{aligned}
&\limsup\limits_{(x,y)\stackrel{\text{gph}~S|_\Omega}{\longrightarrow}(x_3,y_3)}
\frac{\langle(\hat x^*,\hat y^*),(x,y)-(x_3,y_3)\rangle}{\|(x,y)-(x_3,y_3)\|_{X\times Y}}\\
=&\limsup\limits_{(x,y)\stackrel{\text{gph}~S|_\Omega}{\longrightarrow}(x_3,y_3)}
\frac{\langle (x^*+(\varrho+\eta)u_1+\varrho u_2,y^*-2\eta v_1),(x,y)-(x_3,y_3)\rangle}{\|(x,y)-(x_3,y_3)\|_{X\times Y}}\\
\leq&\limsup\limits_{(x,y)\stackrel{\text{gph}~S|_\Omega}{\longrightarrow}(x_3,y_3)}
\frac{\langle(x^*,y^*),(x,y)-(x_3,y_3)\rangle}{\|(x,y)-(x_3,y_3)\|_{X\times Y}}+4\varrho\\
\leq&\gamma\leq\gamma\|(\hat x^*,\hat y^*)\|_{X^*\times Y^*},
\end{aligned}
$$
which justifies $(\hat x^*,\hat y^*)\in \widehat N_\gamma^c\big((x_3,y_3);\operatorname{gph} S|_\Omega\big)$. Combining the latter with the estimate in \eqref{>kappa} contradicts \eqref{nece1} and hence completes the proof of the claimed sufficiency.
\end{proof}

Unifying now the above lemmas brings us to the complete {\it neighborhood characterizations} of the relative Lipschitz-like property of set-valued mappings with the {\it precise calculation} of the exact Lipschitzian bound. 

\begin{theorem}\label{equi} Let $S:X\rightrightarrows Y$ be a set-valued mapping between reflexive Banach spaces, and let $\Omega\subset X$ be a closed and convex set. Assume that the set $\operatorname{gph} S|_\Omega$  is closed around some point  $(\bar x,\bar y)\in\operatorname{gph} S|_\Omega$. Then given $\varepsilon_0\in(0,1)$ and $\kappa_0\geq0$ with $\varepsilon_0=\frac{1}{\sqrt{\kappa_0^2+1}}$, the following statements are equivalent:

{\bf(ii)} For any $\kappa>\kappa_0$, $S$ has the Lipschitz-like property relative to $\Omega$ with constant $\kappa$.

{\bf(ii)} For any $\varepsilon\in(0,\varepsilon_0)$, there exists $\delta>0$ such that \eqref{TX0=0} holds.

{\bf(iii)} For any $\kappa>\kappa_0$, there exist $\nu\in(0,\nu_0)$ and $\delta>0$ such that \eqref{nece1} holds.\\[0.5ex]
Moreover, the exact Lipschitz bound of $S$ relative to $\O$ around $(\ox,\oy)$ is calculated by
\begin{eqnarray}
\operatorname{lip}_\Omega S(\bar x|\bar y)\hskip-0.2cm&=&\hskip-0.2cm\lim\limits_{\delta\to0}\inf
\bigg\{\sqrt{\frac{1}{\varepsilon^2}-1}~\bigg|~\widehat
N_\varepsilon^c\big((x,y);\operatorname{gph} S|_\Omega\big)\times[J(T(x;\Omega))\times\{0\}]=\{(0,0)\}, \nonumber\\
~&&\hskip1cm \mbox{ for all }\;(x,y)\in\operatorname{gph} S|_\Omega\times[(\bar x+\delta B_X)\times (\bar y+\delta B_Y)]\}\bigg\}. \label{x-x-m}
\end{eqnarray}
\end{theorem}

As we see, the exact bound formula (\ref{x-x-m}) makes use of the parameter $\varepsilon$ in the $\varepsilon$-regular normal cone. An explanation of (\ref{x-x-m}) can be as follows. Let $\kappa_0:= \operatorname{lip}_\Omega S(\bar x|\bar y) \geq 0$. We know that $\ve$ in the definition of $\widehat N_\varepsilon^c\big((x,y);\operatorname{gph} S|_\Omega\big)$ cannot be greater than 1, since otherwise the normal cone is the entire space. When $\delta$ in (\ref{x-x-m}) is getting smaller, so is the neighborhood $(\bar x+\delta B_X)\times (\bar y+\delta B_Y)$. Thus the range of the $\ve$'s that satisfies the condition $\widehat N_\varepsilon^c\big((x,y);\operatorname{gph} S|_\Omega\big)\times[J(T(x;\Omega))\times\{0\}]=\{(0,0)\}$ is getting larger. Therefore, the ``inf" is a decreasing function of $\delta$ with a reachable lower bound $\kappa_0$.

Theorem~\ref{equi} tells us that condition \eqref{TX0=0} provides a complete neighborhood characteristics of the relative Lipschitz-like property of set-valued mappings when the constraint set $\O$ is closed and convex. However, such s characterization {\it cannot be achieved} if the $\ve$-regular normal cone \eqref{e-cone} is replaced by the set of $\ve$-normals \eqref{e-set}. The following example confirms this statement even in the case
of the simplest single-valued mapping $S\colon\R\to\R$ and the closed interval $\O$ in $\R$.

\begin{example}\label{exa:e-dif}
{\rm Let $S:\mathbb R\to\mathbb R$ be defined by $S(x):=x$ for all $x\in\mathbb R$, and let $\Omega:=[0,1]$. Clearly, $S$ has the Lipschitz-like property relative to $\Omega$ around $(0,0)$. In this case, $J$ is the identity mapping, $T(0;\Omega)=[0,\infty)$, and 
$$
\widehat N_\varepsilon((0,0);\operatorname{gph} S|_\Omega)=\big\{(x^*,y^*)\in\R^2~\big|~x^*+y^*\leq\sqrt 2\varepsilon\big\}
$$
for any $\ve>0$. Therefore, we have the expression
$$
\widehat N_\varepsilon((0,0);\operatorname{gph} S|_\Omega)\cap[T(0;\Omega)\times\{0\}]=[0,\sqrt2\varepsilon]\times\{0\},
$$
which tells us that the counterpart of \eqref{TX0=0} fails for the set of $\ve$-normals \eqref{e-set}.}
\end{example}\vspace*{-0.15in}

\section{Pointbased Coderivative Characterizations of the Relative Lipschitz-like Property}\label{sec:point}

The usage of the neighborhood characterizations of the relative Lipschitz-like property obtained in Section~\ref{sec:chara} may be challenging to apply, since they involve the consideration of neighborhood points. Much more preferable are {\it pointbased criteria}, which are expressed precisely at the reference point. To derive such pointbased results, in this section we introduce new coderivative constructions, which enjoy the major pointbased calculus rules developed in the subsequent Sections~\ref{sec:chain} and \ref{sec:sum}.\vspace*{0.02in}

Recall first the notions of (sequential) outer and inner limits for multifunctions acting between Banach spaces that are used in what follows.

\begin{Def}\label{limits}
Let $S:X\rightrightarrows Y$ be a set-valued mapping between Banach spaces, and let $\bar x\in{\rm dom}\,S:=\{x\in X\;|\;S(x)\ne\emptyset\}$. Then we have:

{\bf (i)} The {\sc strong} and {\sc weak}$^*$ $($if $Y=Z^*$ for some Banach space Z$)$ {\sc outer limits} are defined, respectively, by
$$
 \begin{aligned}
 &s\mbox{-}\Limsup_{x\to\ox}S(x):=\big\{y\in Y~\big|~\exists x_k\to\bar x~\operatorname{and}~y_k\in S(x_k)~\operatorname{such~that}~y_k\to y\big\},\\
 &w^*\mbox{-}\Limsup_{x\to\bar x}S(x):=\big\{y\in Y~\big|~\exists x_k\to\bar x~\operatorname{and}~y_k\in S(x_k)~\operatorname{such~that}~y_k\rightharpoonup^* y\big\}.
 \end{aligned}
 $$
 
 {\bf(ii)} The {\sc strong inner limit} is described by
 $$
s\mbox{-}\Liminf_{x\to\bar x}
S(x):=\big\{y\in Y~\big|~\forall x_k\to\bar x,~\exists y_k\in S(x_k)~\operatorname{such~that}~y_k\to y\big\}.
$$
\end{Def}

Next we define the coderivative notions needed for our characterizations. Note that the $\ve$-coderivative construction from Definition~\ref{coderivatives}(i) was actually used in Section~\ref{sec:chara}, where we preferred to deal directly with the $\ve$-regular normal cone. Similarly to \cite{Mordukhovich2023}, observe that, in comparing with the recent {\it projectional coderivative} studied in \cite{Yang2021} via projections of normals to the tangent cone of the constraint set, here we are motivated by the fact that the intersection of the $\varepsilon$-regular normal cone with the contingent cone would be much easier to deal with than to find projections.

\begin{definition}\label{coderivatives}
Let $S:X\rightrightarrows Y$ be a set-valued mapping between Banach spaces, let $\Omega$ be a nonempty subset of $X$, and let $(\bar x,\bar y)\in \operatorname{gph} S|_\Omega$. We have the following constructions:

{\bf(i)} Given $\varepsilon\geq0$, the {\sc $\varepsilon$-conic regular contingent coderivative} of $S$ {\sc relative to} $\Omega$ at $(\bar x,\bar y)$ is defined as a multifunction $\widehat D_{\Omega,\varepsilon}^{c*} S(\bar x|\bar y):Y^*\rightrightarrows X^*$ with the values
\begin{equation}\label{harnepsilion'}
\begin{array}{ll}
\widehat D_{\Omega,\varepsilon}^{c*} S(\bar x|\bar y)(y^*):=\big\{x^*\in X^*~\big|&(x^*,-y^*)\in\widehat N_\varepsilon^c\big((\bar x,\bar y);\operatorname{gph} S|_\Omega\big)\\
&\cap\big[J\big(T(\bar x;\Omega)\big)\times Y^*\big]\big\}.
\end{array}
\end{equation}

{\bf(ii)} The {\sc conic normal contingent coderivative} of $S$ {\sc relative to} $\Omega$ at $(\bar x,\bar y)$ is a multifunction $\widetilde D_\Omega^{c*}S(\bar x|\bar y):Y^*\rightrightarrows X^*$ defined by
\begin{equation}\label{normalcoder}
\widetilde D_\Omega^{c*}S(\bar x|\bar y)(\bar y^*):=w^*\mbox{-}\Limsup_{\stackrel{\stackrel{(x,y)\stackrel{\operatorname{gph} S|_\Omega}{\longrightarrow}(\bar x,\bar y)}{y^*\rightharpoonup^* \bar y^*}}{\varepsilon\downarrow0}}
\widehat D_{\Omega,\varepsilon}^{c*} S(x|y)(y^*).
\end{equation}
That is, \eqref{normalcoder} is the collection of such $\bar x^*\in X^*$ for which there exist sequences $\varepsilon_k\downarrow 0$, $(x_k,y_k)\stackrel{\operatorname{gph} S|_\Omega}{\longrightarrow}(\bar x,\bar y)$, and $(x_k^*,y_k^*)\rightharpoonup^*(\bar x^*,\bar y^*)$ with
\begin{equation}\label{cod-seq}
(x_k^*,-y_k^*)\in\widehat N^c_{\varepsilon_k}\big((x_k,y_k);\operatorname{gph} S|_\Omega\big)\cap\big[J\big(T(x_k;\Omega)\big)\times Y^*\big],\quad k\in\mathbb N.
\end{equation}

{\bf(iii)} The {\sc conic mixed contingent coderivative} of $S$ {\sc relative to} $\Omega$ at $(\bar x,\bar y)$ is a multifunction $D_\Omega^{c*}S(\bar x|\bar y):Y^*\rightrightarrows X^*$ defined by
\begin{equation}\label{mixedcoder}
D_\Omega^{c*}S(\bar x|\bar y)(\bar y^*):=w^*\mbox{-}\Limsup_{\stackrel{\stackrel{(x,y)\stackrel{\operatorname{gph} S|_\Omega}{\longrightarrow}(\bar x,\bar y)}{y^*\longrightarrow\bar y^*}}{\varepsilon\downarrow 0}}
\widehat D_{\Omega,\varepsilon}^{c*}S(x|y)(y^*).
\end{equation}
That is, \eqref{mixedcoder} is the collection of such $\bar x^*\in X^*$ for which there exist sequences $\varepsilon_k\downarrow0$, $(x_k,y_k)\stackrel{\operatorname{gph} S|_\Omega}{\longrightarrow}(\bar x,\bar y)$, $y_k^*\to\bar y^*$ and $x_k^*\rightharpoonup^*\bar x^*$ with $(x^*_k,y^*_k)$ satisfying \eqref{cod-seq}.

{\bf(iv)} We say that $S$ is {\sc coderivatively normal relative to} $\O$ at $(\ox,\oy)$ if 
\begin{equation*}
\|D^{c*}_\O S(\ox,\oy)\|=\|\Tilde D^{c*}_\O S(\ox,\oy)\|
\end{equation*}
via the coderivative norms as positively homogeneous mappings \eqref{norm}.
\end{definition}
 
Note that the $\ve$-conic regular contingent coderivative from Definition~\ref{coderivatives}(i) differs from the $\ve$-contingent relative coderivative from \cite[Definition~2.4(i)]{Mordukhovich2023} by using the $\ve$-regular normal cone \eqref{e-cone} instead of the $\ve$-normal set \eqref{e-set}. As shown in the previous Section~\ref{sec:chara}, the usage of the new $\ve$-coderivative allows us to establish {\it neighborhood characterizations} of the relative Lipschitz-like property of set-valued mappings, which is not the case for the $\ve$-construction from \cite{Mordukhovich2023}; see Example~\ref{exa:e-dif}.

On the hand, the next lemma tells us that the limiting constructions from Definition~\ref{coderivatives}(ii,iii) agrees with the corresponding limiting coderivatives from \cite[Definition~2.4]{Mordukhovich2023} in the general Banach space setting. Note that this does not diminish the value of our new coderivative constructions. First of all, they are simpler to deal with and allow us to essentially simplify the proofs of the stability results from \cite{Mordukhovich2023}; see, e.g., Theorem~\ref{criterion} below. Furthermore, it is most important for variational theory and applications that the newly introduced limiting coderivatives enjoy the fundamental calculus rules derived in Sections~\ref{sec:chain} and \ref{sec:sum}. These results are also new for the coderivative constructions from \cite{Mordukhovich2023} for which calculus rules have not been developed earlier. 

\begin{lemma}\label{equal} Let $S:X\rightrightarrows Y$ be a set-valued mapping between Banach spaces, let $\O\subset X$, and let $(\bar x,\bar y)\in \operatorname{gph} S|_\Omega$. The following hold:

{\bf(i)} We have $x^*\in D_\Omega^{c*}S(\bar x|\bar y)(y^*)$ if and only if there exist sequences $\varepsilon_k\downarrow0$, $(x_k,y_k)\stackrel{\operatorname{gph} S|_\Omega}{\longrightarrow}(\bar x,\bar y)$, and $(x_k^*,-y^*)\in\widehat N_{\varepsilon_k}\big((x_k,y_k);\operatorname{gph} S|_\Omega\big)\cap\big[J\big(T(x_k;\Omega)\big)\times Y^*\big]$ such that
$y_k^*\to y^*$ and $x_k^*\rightharpoonup^* x^*$ as $k\to\infty$.

{\bf(ii)} We have $x^*\in\widetilde D_\Omega^{c*}S(\bar x|\bar y)(y^*)$ if and only if there exist sequences $\varepsilon_k\downarrow0$, $(x_k,y_k)\stackrel{\operatorname{gph} S|_\Omega}{\longrightarrow}(\bar x,\bar y)$, and $(x_k^*,-y_k^*)\in\widehat N_{\varepsilon_k}\big((x_k,y_k);\operatorname{gph} S|_\Omega\big)\cap\big[J\big(T(x_k;\Omega)\big)\times Y^*\big]$ such that
$y_k^*\rightharpoonup^* y^*$ and $x_k^*\rightharpoonup^* x^*$ as $k\to\infty$.
\end{lemma}
\begin{proof} To verify (i), consider the set
\begin{equation}\label{K}
\begin{array}{ll}
K:=\big\{&x^*\in X^*~\big|~\operatorname{there~exist}~\varepsilon_k\downarrow0, (x_k,y_k)\stackrel{\operatorname{gph} S|_\Omega}{\longrightarrow}(\bar x,\bar y),\\
&\operatorname{and}~(x_k^*,-y_k^*)\in\widehat N_{\varepsilon_k}\big((x_k,y_k);\operatorname{gph} S|_M\big)\cap\big[J\big(T(x_k;\Omega)\big)\times Y^*\big]\\
&\operatorname{such~that}y_k^*\rightharpoonup^* y^*\;\mbox{ and }\;x_k^*\rightharpoonup^* x^*\big\}.
\end{array}
\end{equation}
Pick any $x^*\in D_\Omega^{c*}S(\bar x|\bar y)(y^*)$ and find by definition \eqref{mixedcoder} sequences 
\begin{equation}\label{equal1}
\varepsilon_k\downarrow0,~~ (x_k,y_k)\stackrel{\operatorname{gph} S|_\Omega}{\longrightarrow}(\bar x,\bar y),\;\mbox{ and }
\end{equation}
\begin{equation}\label{equal2}
(x_k^*,-y_k^*)\in\widehat N_{\varepsilon_k}^c\big((x_k,y_k);\operatorname{gph} S|_\Omega\big)\cap\big[J\big(T(x_k;\Omega)\big)\times Y^*\big]
\end{equation}
for which we have the convergence
\begin{equation}\label{equal3}
y_k^*\to y^*~~ \operatorname{and}~~ x_k^*\rightharpoonup^* x^*\;\mbox{ as }\;k\to\infty.
\end{equation}
It follows from \eqref{equal3} by the uniform boundedness principle that there is $L>0$ with
$$
\|(x_k^*,-y_k^*)\|_{X^*\times Y^*}\leq L\;\mbox{ for all }\; k\in\N.
$$
This implies therefore by definition \eqref{e-cone} that 
$$
\limsup\limits_{(x,y)\stackrel{\operatorname{gph} S|_\Omega}{\longrightarrow}(x_k,y_k)}\frac{\langle(x_k^*,-y_k^*),(x,y)-(x_k,y_k)\rangle}{\|(x,y)-(x_k,y_k)\|_{X\times Y}}\leq\varepsilon_k\|(x_k^*,-y_k^*)\|_{X^*\times Y^*}\leq L\varepsilon_k.
$$
Using the latter together with \eqref{equal2} ensures that
\begin{equation}\label{equal4}
(x_k^*,-y_k^*)\in\widehat N_{L\varepsilon_k}\big((x_k,y_k);\operatorname{gph} S|_\Omega\big)\cap\big[J\big(T(x_k;\Omega)\big)\times Y^*\big],\quad k\in\N.
\end{equation}
Combining now \eqref{K}, \eqref{equal1}, \eqref{equal3}, and \eqref{equal4} brings us to $x^*\in K$ and thus justifies the inclusion $D_\Omega^{c*}S(\bar x|\bar y)(y^*)\subset K$.\vspace*{0.02in}

To verify the opposite inclusion, fix $x^*\in K$ and find by definition \eqref{K} sequences
\begin{equation}\label{equal5}
\varepsilon_k\downarrow 0, ~~(x_k,y_k)\stackrel{\operatorname{gph} S|_\Omega}{\longrightarrow}(\bar x,\bar y),\;\mbox{ and }
\end{equation}
\begin{equation}\label{equal6}
(x_k^*,-y^*)\in\widehat N_{\varepsilon_k}\big((x_k,y_k);\operatorname{gph} S|_\Omega\big)\cap\big[J\big(T(x_k;\Omega)\big)\times Y^*\big]
\end{equation}
satisfying the convergence properties in \eqref{equal3}. In the case where $(x_k^*,-y_k^*)\to(0,0)$, we get $(x^*,y^*)=(0,0)$ and obviously have $x^*\in D_\Omega^{c*}S(\bar x|\bar y)(y^*)$. Thus assume without loss of generality that there exists $M>0$ such that
$$
\|(x_k^*,-y_k^*)\|_{X^*\times Y^*} \geq M\;\mbox{ for all }\;k\in\N,
$$
which tells us by \eqref{equal6} and definition \eqref{e-set} that
$$
\limsup\limits_{(x,y)\stackrel{\operatorname{gph} S|_\Omega}{\longrightarrow}(\bar x,\bar y)}\frac{\langle (x_k^*,-y_k^*),(x,y)-(x_k,y_k)\rangle}{\|(x,y)-(x_k,y_k)\|_{X\times Y}}\leq\varepsilon_k\leq\frac{\varepsilon_k}{M}\|(x_k^*,-y_k^*)\|_{X^*\times Y^*}.
$$
By \eqref{e-cone}, the latter implies that
$$
(x_k^*,-y_k^*)\in\widehat N^c_{\varepsilon_k/M}\big((x_k,y_k);\operatorname{gph} S|_\Omega\big).
$$
Employing \eqref{equal6} again brings us to the inclusion
\begin{equation}\label{equal8}
(x_k^*,-y_k^*)\in\widehat N^c_{\varepsilon_k/M}\big((x_k,y_k);\operatorname{gph} S|_\Omega\big)\cap\big[J\big(T(x_k;\Omega)\big)\times Y^*\big],
\end{equation}
which shows that $x^*\in D_\Omega^{c*}S(\bar x|\bar y)(y^*)$ by combining \eqref{equal3}, \eqref{equal5}, and \eqref{equal8}.

(ii) The proof of assertion (ii) is similar to (i).
\end{proof}

To proceed further with deriving pointbased characterizations of the relative Lipschitz-like property, we formulate the following {\it relative partial sequentially normal compactness} (relative PSNC) condition in the vein of \cite{Mordukhovich2006} for unconstrained multifunction and of \cite{Mordukhovich2023} for constrained ones. Note that the latter condition differs from the one given below by the usage of the $\ve$-normal set \eqref{e-set} instead  of the $\ve$-regular normal cone \eqref{e-cone}. It follows from Lemma~\ref{equal} that these two PSNC properties are in fact equivalent. Note that they are obviously fulfilled if $X$ is finite-dimensional.

\begin{definition}\label{psnc} Given a set-valued mapping $S\colon X\tto Y$ between Banach space and an nonempty set $\O\subset X$, it is said that $S$ has the {\sc relative PSNC property} with respect to $\O$ at $(\bar x,\bar y)\in\operatorname{gph}S|_\Omega$ if for any sequence $(\varepsilon_k,x_k,y_k,x_k^*,y_k^*)\in [0,\infty)\times\operatorname{gph} S|_\Omega\times X^*\times Y^*$ satisfying the relationships
\begin{equation}\label{PSNC}
\begin{aligned}
&\varepsilon_k\downarrow 0,\;(x_k,y_k)\to(\bar x,\bar y),~(x_k^*,-y_k^*)\in \widehat N_{\varepsilon_k}^c\big((x_k,y_k);\operatorname{gph} S|_\Omega\big)\cap[J\big(T(x_k;\Omega)\big)\times Y^*],\\
&x_k^*\rightharpoonup^*0,\;\operatorname{and}~y_k^*\to 0\;\mbox{ as }\;k\to\infty,
\end{aligned}
\end{equation}
we have the strong $X^*$-convergence $\|x_k^*\|_{X^*}\to 0$ as $k\to \infty$.
\end{definition}

The next theorem presents pointbased necessary conditions, sufficient conditions, and characterizations of the relative Lipschitz-like property of multifunctions with evaluations of the exact relative Lipschitzian bound at the point in question via the coderivative norms as positive homogeneous mappings \eqref{norm}. 
By taking into account Lemma~\ref{equal} and the discussions above, the characterization part can be derived from \cite[Theorem~3.6]{Mordukhovich2023}.

\begin{theorem}\label{criterion}
Let $\Omega\subset X$ and $(\bar x,\bar y)\in\operatorname{gph} S|_\Omega$, where $S:X\rightrightarrows Y$ be a set-valued mapping between reflexive Banach spaces. Consider the following statements:

{\bf(i)} $S$ has the Lipschitz-like property relative to $\Omega$ around $(\bar x,\bar y)$.

{\bf(ii)} $S$ is PSNC at $(\bar x, \bar y)$ relative to $\Omega$ and $\|D_\Omega^{c*}S(\bar x|\bar y)\|<\infty$.

{\bf(iii)} $S$ is PSNC at $(\bar x,\bar y)$ relative to $\Omega$ and $D_\Omega^{c*}S(\bar x|\bar y)(0)=\{0\}$.\\[0.5ex] 
Then we have implications {\rm(i)}$\Longrightarrow${\rm(ii)}$\Longrightarrow${\rm(iii)}. If in addition the set $\O$ is closed and convex and the set $\operatorname{gph} S$ is locally closed around $(\bar x,\bar y)$, then we also have implication  {\rm(iii)$\Longrightarrow$(i)}. Moreover, if in the latter setting $\dim X<\infty$, that the exact Lipschitzian bound of $S$ relative to $\Omega$ around $(\bar x,\bar y)$ satisfies the estimate
\begin{equation}\label{inequ-estimate}
\operatorname{lip}_\Omega S(\bar x,\bar y)\leq\|\widetilde D_\Omega^{c*}S(\bar x|\bar y)\|,
\end{equation}
while the lower bound estimate
\begin{equation}\label{est}
\|D_\Omega^{c*}S(\bar x|\bar y)\|\le\operatorname{lip}_\Omega S(\bar x,\bar y)
\end{equation}
holds for $S\colon X\tto Y$ between any reflexive Banach spaces $X$ and $Y$ relative to an arbitrary set $\O$. Finally, if we add to the above assumptions on the fulfillment of \eqref{inequ-estimate} that $S$ is coderivative normal relative to $\Omega$ at $(\bar x,\bar y)$, then 
\begin{equation}\label{equ-estimate}
\operatorname{lip}_\Omega S(\bar x,\bar y)=\|D_\Omega^{c*}S(\bar x|\bar y)\|= \|\widetilde D_\Omega^{c*}S(\bar x|\bar y)\|.
\end{equation}
\end{theorem}
\begin{proof}
First we simultaneously verify implication (i)$\Longrightarrow$(ii) and the lower estimate \eqref{est} of the exact Lipschitzian bound in the general setting of the theorem. Assume that $S$ enjoys the Lipschitz-like property relative to $\Omega$ around $(\bar x,\bar y)$ with constant $\kappa$ and take a sequence $(\varepsilon_k,x_k,y_k,x_k^*,y_k^*)\in (0,1)\times\operatorname{gph} S|_\Omega\times X^*\times X^*\times Y^*$ satisfying \eqref{PSNC}. It follows from Lemma~\ref{Necessity} that $\|(x_k^*)\|_{X^*}\leq(\kappa+1)\|y_k^*\|_{Y^*}$ for large $k\in\N$, and hence $x_k^*\to 0$ as $k\to\infty$, which yields the PSNC property of $S$ at $(\bar x,\bar y)$ relative to $\Omega$.

Further, pick any $(x^*, y^*)\in X^*\times Y^*$ with $x^*\in D_\Omega^*S(\bar x|\bar y)(y_*)$ and find by the coderivative definition such sequences $\varepsilon_k\dn 0$, $(x_k,y_k)\stackrel{\operatorname{gph} S|_\Omega}{\longrightarrow}(\bar x,\bar y)$, and $(x_k^*,-y_k^*)\in\widehat N^c_{\varepsilon_k}\big((x_k,y_k);\operatorname{gph} S|_\Omega\big)$ that $y_k^*\to y^*$ and $x_k^*\rightharpoonup^* x^*$ as $k\to\infty$. Denoting $\kappa_k:=\kappa+k^{-1}$, it follows from Lemma~\ref{Necessity} that $\|x_k^*\|_{X^*}\leq\kappa_k\|y_k^*\|_{Y^*}$ when $k$ is sufficiently large. Therefore,
$$
\|x^*\|_{X^*}\leq\liminf\limits_{k\to\infty}\|x_k^*\|_{X_k^*}\leq\liminf\limits_{k\to\infty}\kappa_k\|y_k^*\|_{Y^*}=\kappa\|y^*\|_{Y^*},
$$
which tells us that $\|D_\Omega^*S(\bar x|\bar y)\|\leq\kappa$ and thus verifies (ii) together with the lower estimate of the exact Lipschitzian bound in \eqref{est}. Implication (ii)$\Longrightarrow$(iii) is obvious.

Next we proceed with the justification of (iii)$\Longrightarrow$(i) under the additional assumptions made. Arguing by contradiction, suppose that $S$ does not have the Lipschitz-like property relative to $\Omega$ around $(\bar x,\bar y)$. It follows from Lemma~\ref{Sufficiency} that there exist $\varepsilon_k\dn 0$, $(x_k,y_k)\stackrel{\operatorname{gph} S|_\Omega}{\longrightarrow}(\bar x,\bar y)$, and $(x_k^*,-y_k^*)\in\widehat N^c_{\varepsilon_k}\big((x_k,y_k);\operatorname{gph} S|_\Omega\big)$ for which the estimate
$$
\|x_k^*\|_{X^*}>k\|y_k^*\|_{Y^*}\;\mbox{ for all }\;k\in\N
$$
is satisfied. We may assume without generality that $\|x_k^*\|_{X^*}=1$ and hence get that $y_k^*\to 0$ as $k\to\infty$. It follows from the reflexivity of $X$ that the dual ball $B_{X^*}$ is weak$^*$ compact in $X^*$, and hence there exists $x^*\in X^*$ such that $x_k^*\rightharpoonup^* x^*$ along a subsequence. Applying the condition $D_\Omega^*(\bar x|\bar y)(0)=\{0\}$ gives us $x^*=0$, i.e., $x_k^*\rightharpoonup^*0$. By the imposed PSNC property of $S$ at $(\bar x,\bar x)$ relative to $\Omega$, this implies that $x_k^*\to 0$ as $k\to\infty$ along the selected subsequence, which contradicts $\|x_k^*\|_{X^*}=1$ and thus gives us (i).

It remains to verify the upper bound estimate assuming in addition that $\dim X<\infty$. Fix any $\kappa>\|\widetilde D_\Omega^*S(\bar x|\bar y)\|$ and show that we can find $\delta>0$ so that \eqref{nece1} holds for any small $\varepsilon>0$. Supposing the contrary yields the existence of sequences $\varepsilon_k\downarrow0$, $(x_k,y_k)\stackrel{\operatorname{gph} S|_\Omega}{\longrightarrow}(\bar x,\bar y)$ as $k\to\infty$, and
$$
(x_k^*,y_k^*)\in\widehat N^c_{\varepsilon_k}\big((x_k,y_k);\operatorname{gph} S|_\Omega\big)\cap\big[J\big(T(x_k;\Omega)\big)\times Y^*\big]
$$
along which $\|x_k^*\|_{X^*}>\kappa\|y_k^*\|_{Y_*}$ whenever $k\in\N$. We may assume that $\{x_k^*\}$ is bounded in $X$. Since $\dim X<\infty$ and $Y$ is reflexive, there is a pair $(x^*,y^*)\in X^*\times Y^*$ such that $ x_k^*\to x^*$ and $y_k^*\rightharpoonup^*y^*$ along some subsequences. Remembering that the norm function is continuous in finite dimensions and lower semicontinuous in the weak$^*$ topology of duals to Banach spaces, we get $x^*\in\widetilde D_\Omega^*S(\bar x|\bar y)(y^*)$ and
$$
\|y^*\|_{Y^*}\leq\liminf\limits_{k\to\infty}\|y_k^*\|_{Y^*}
\leq\kappa^{-1}=\kappa^{-1}\|x^*\|_{X^*},
$$
which contradicts the condition $\kappa>\|\widetilde D_\Omega^*S(\bar x|\bar y)\|$ and thus verifies \eqref{inequ-estimate}. Equality \eqref{equ-estimate} is a consequence of \eqref{inequ-estimate}, \eqref{est}, and the assumed coderivative normality.
\end{proof}\vspace*{-0.15in}

\section{Chain Rules for Conic Contingent Coderivatives}\label{sec:chain}

In this section, we show that both conic mixed and normal contingent coderivatives from Definition~\ref{coderivatives}(ii,iii)  for constrained multifunctions between reflexive Banach spaces
satisfy {\it pointbased chain rules} under a mild qualification condition. This qualification condition is expressed in terms of the relative mixed  contingent coderivative \eqref{mixedcoder} and holds, in particular, under the relative well-posedness properties of the multifunctions in question due to the coderivative characterizations of Section~\ref{sec:point}.
Since chain rules have never been developed for coderivative constructions from \cite{Mordukhovich2023}, the obtained results are also new for those constructions. The proofs below mostly revolve around the (approximate) {\it extremal principle} of variational analysis; see \cite{Mordukhovich2006}.\vspace*{0.02in} 

We begin with evaluating the $\varepsilon$-regular normal cone to inverse images of sets under single-valued continuous mappings that holds in any Banach spaces.

\begin{lemma}\label{F-1}
Let $f:X\to Y$ be a single-valued continuous mapping between Banach spaces, and let $\Theta$ be a nonempty subset of $Y$. Take some $\bar x\in X$ and denote $\bar y:=f(\bar x)\in\Theta$. Then for any $\varepsilon>0$, we have the inclusion
$$
\widehat N_\varepsilon^c\big(\bar x;f^{-1}(\Theta)\big)\subset\big\{x^*\in X^*~\big|~(x^*,0)\in\widehat N_\varepsilon^c\big((\bar x,\bar y);(X\times\Theta)\cap\operatorname{gph}f\big)\big\}.
$$
\end{lemma}
\begin{proof} Fix $x^*\in\widehat N_\varepsilon^c(\bar x;f^{-1}(\Theta))$ and get by
Definition \eqref{e-cone} of the $\ve$-regular normals that 
$$
\limsup\limits_{x\stackrel{f^{-1}(\Theta)}{\longrightarrow}\bar x}\frac{\langle x^*,x-\bar x\rangle}{\|x-\bar x\|_X}\le\varepsilon\|x^*\|_{X^*}.
$$
The continuity of $f$ and the convergence $x\stackrel{f^{-1}(\Theta)}{\longrightarrow}\bar x$ yields $f(x)\stackrel{\Theta}{\longrightarrow} f(\bar x)=\bar y$. Thus
$$
\begin{aligned}
\limsup\limits_{(x,y)\stackrel{[X\times\Theta]\cap\operatorname{gph}f}{\longrightarrow}(\bar x,\bar y)}\frac{\langle(x^*,0),(x,y)-(\bar x,\bar y)\rangle}{\|(x,y)-(\bar x,\bar y)\|_{X\times Y}}
\leq&\limsup\limits_{x\stackrel{f^{-1}(\Theta)}{\longrightarrow}\bar x}\frac{\langle (x^*,0),(x,f(x))-(\bar x,f(\bar x))\rangle}
{\sqrt{\|x-\bar x\|_X^2+\|f(x)-f(\bar x)\|_Y^2}}\\
\le&\limsup\limits_{x\stackrel{f^{-1}(\Theta)}{\longrightarrow}\bar x}\frac{\langle x^*,x-\bar x\rangle}{\|x-\bar x\|_X}
\leq\varepsilon\|x^*\|_{X^*},
\end{aligned}
$$
which yields $(x^*,0)\in\widehat N_\varepsilon^c\big((\bar x,\bar y);(X\times\Theta)\cap\operatorname{gph}f\big)$ and verifies the claimed inclusion.
\end{proof}

Next we recall the notions of inner semicontinuity and inner semicompactness for set-valued mappings relative to sets. These notions are extensions of those from \cite{Mordukhovich2006} to the case of constrained multifunctions.

\begin{definition}\label{inner-semi}
Let $S:X\rightrightarrows Y$ be a set-valued mapping between Banach spaces, $\Omega\subset X$, $\bar x\in \operatorname{dom}S\cap \Omega$, and $(\bar x,\bar y) \in \operatorname{gph}S|_\Omega$. Then we say that:
\begin{description}
\item[(i)] $S$ is {\sc inner semicontinuous relative to} $\Omega$ at $(\bar x,\bar y)$ if for every sequence $x_k\stackrel{\operatorname{dom}S\cap \Omega}{\longrightarrow}\bar x$ there exists $y_k\in S(x_k)$ converging to $\bar y$ as $k\to\infty$.
\item[(ii)] We say that $S$ is {\sc inner semicompact relative to} $\Omega$ at $(\bar x,\bar y)$ if for every sequence $x_k\stackrel{\operatorname{dom} S\cap\Omega}{\longrightarrow}\bar x$ there exists $y_k\in S(x_k)$ that contains a convergent subsequence.
\end{description}
\end{definition}

To establish general chain rules for the relative contingent coderivatives from Definition~\ref{coderivatives}(ii,iii), we now introduce an auxiliary coderivative construction and a related PSNC property used in deriving the desired results.

\begin{definition}\label{mirrorcc}
Let $S:X\rightrightarrows Y$ be a set-valued mapping between Banach spaces, and let $\Theta$ be a nonempty subset of $Y$. Denote ${\mathfrak F}(x):=S(x)\cap\Theta$ and fix $(\bar x,\bar y)\in \operatorname{gph}\mathfrak{F}$. Then:

{\bf(i)} The {\sc complete mirror contingent coderivative} of $S$ {\sc relative to} $\Theta$ at $(\bar x,\bar y)$ is the multifunction $D_\Theta^{*-1}S(\bar x|\bar y):Y^*\rightrightarrows X^*$ meaning that $x^*\in D_\Theta^{*-1}S(\bar x|\bar y)(y^*)$ if and only if there exist sequences $\varepsilon_k\downarrow 0$, $(x_k,y_k)\stackrel{\operatorname{gph} \mathfrak F}{\longrightarrow}(\bar y,\bar x)$ as $k\to\infty$, and
$$
(x_k^*,-y_k^*)\in\widehat N_{\varepsilon_k}^c\big((x_k,y_k);\operatorname{gph} \mathfrak F\big)\cap\big[X^*\times J\big(T(y_k;\Theta)\big)\big],\quad k\in\mathbb N,
$$
such that $x_k^*\rightharpoonup^* x^*$ and $y_k^*\to y^*$ as $k\to\infty$.

{\bf(ii)} We say that $S$ is {\sc completely mirror PSNC relative to} $\Omega$ at $(\bar x,\bar y)\in\operatorname{gph}S|_\Omega$ if for any sequence $(\varepsilon_k,x_k,y_k,x_k^*,y_k^*)\in (0,1)\times\operatorname{gph} \mathfrak F\times X^*\times Y^*$ satisfying
$$
\begin{aligned}
&\varepsilon_k\downarrow0,(x_k,y_k)\to(\bar x,\bar y),~(x_k^*,-y_k^*)\in \widehat N_{\varepsilon_k}^c\big((x_k,y_k);\operatorname{gph} \mathfrak F\big)\cap\big[X^*\times J\big(T(y_k;\Theta)\big)\big],\\
&x_k^*\rightharpoonup^*0,\;\operatorname{and}~y_k^*\to 0\;\mbox{ as }\;k\to\infty,
\end{aligned}
$$
the strong $X^*$-convergence $\|x_k^*\|_{X^*}\to 0$ as $k\to \infty$ holds.
\end{definition}

Here are our main {\it chain rules} for {\it both relative contingent coderivatives} from Definition~\ref{coderivatives}(ii,iii) that are formulated via the refined constructions from Definition~\ref{mirrorcc}. For brevity, we confine ourselves to the case of the inner semicontinuity of the auxiliary mapping in Theorem~\ref{chainrule}(i), while both properties from Definition~\ref{inner-semi} are used in the coderivative sum rules derived in the next section.

\begin{theorem}\label{chainrule}
Let $S_1:X\rightrightarrows Y$ and $S_2:Y\rightrightarrows Z$ be set-valued mappings between reflexive Banach spaces, and let $\Omega$ be a closed and convex subset of $X$. Consider the composition $S:=S_2\circ S_1\colon X\rightrightarrows Z$ and define the auxiliary multifunction $G\colon X\times Z\rightrightarrows Y$ by
\begin{equation}\label{G-semi}
G(x,z):=S_1(x)\cap S_2^{-1}(z)=\big\{y\in S_1(x)~\big|~z\in S_2(y)\big\}\;\mbox{ for }\;x\in X,\;z \in Z.
\end{equation}
Fixing $(\bar x,\bar z)\in\operatorname{gph}S|_\Omega$ and $\bar y\in G(\bar x,\bar z)$, suppose that $\Omega$ is convex and locally closed around $\bar x$, that the graphs of $S_1$ and $S_2$ are locally closed around $(\bar x,\bar y)$ and
$(\bar y,\bar z)$, respectively, and that $G$ is inner semicontinuous relative to $\Omega\times Z$ at $(\bar x,\bar z,\bar y)$. Then the following assertions are satisfied:

{\bf(i)} Assume that either $S_2$ is PSNC at $(\bar y,\bar z)$ or $S_1^{-1}$ is completely mirror PSNC relative to $\Omega$ at $(\bar y,\bar x)$, and that the relative mixed coderivative qualification condition
\begin{equation}\label{mixedQC1}
D^*S_2(\bar y|\bar z)(0)\cap\big(-D_\Omega^{*-1}S_1^{-1}(\bar y|\bar x)(0)\big)=\{0\}
\end{equation}
holds. Then for all $z^*\in Z^*$ we have the inclusions
\begin{equation}\label{chainrule1}
D_\Omega^{c*}S(\bar x|\bar z)(z^*)\subset \widetilde D_\Omega^{c*} S_1(\bar x|\bar y)\circ D^*S_2(\bar y|\bar z)(z^*),
\end{equation}
\begin{equation}\label{chainrule2}
\widetilde D_\Omega^{c*}S(\bar x|\bar z)(z^*)\subset \widetilde D_\Omega^{c*}S_1(\bar x|\bar y)\circ \widetilde D^*S_2(\bar y|\bar z)(z^*).
\end{equation}

{\bf(ii)} If $S_2:=f\colon Y\to Z$ is single-valued and strictly differentiable at $\bar y$, then we have the enhanced inclusion for the relative mixed contingent coderivative:
\begin{equation}\label{chainrule5}
D_\Omega^{c*}S(\bar x|\bar z)(z^*)\subset D_\Omega^{c*}S_1(\bar x|\bar y)\big(\nabla f(\bar y)^*z^*)\big)\;\mbox{ whenever }\;z^*\in Z^*.
\end{equation}
\end{theorem}
\begin{proof} To justify assertion (i), we just verify \eqref{chainrule1} while observing that the proof of \eqref{chainrule2} is similar and thus can be omitted. To proceed with \eqref{chainrule1}, pick any $x^*\in D_\Omega^{c*}S(\bar x|\bar z)(z^*)$ and find sequences
$\varepsilon_k\downarrow0$, $(x_k,z_k)\stackrel{\operatorname{gph} S|_\Omega}{\longrightarrow}(\bar x,\bar z)$, and
\begin{equation}\label{chain2}
(x_k^*,-z_k^*)\in\widehat N_{\varepsilon_k}^c\big((x_k,z_k);\operatorname{gph} S|_\Omega\big)\cap[J(T(x_k;\Omega)\big)\times Z^*], \quad k\in\mathbb N,
\end{equation}
along which we have the convergence $z_k^*\to z^*~~\operatorname{and}~~x_k^*\rightharpoonup^*x^*$ as $k\to\infty$.
Since $G$ is inner semicontinuous relative to $\Omega\times Z$ at $(\bar x,\bar y,\bar z)$, there are $y_k\in G(x_k,z_k)$ with $y_k\to \bar y$.

Let $C:=\{(x,y,z)\in X\times Y\times Z~|~y\in S_1|_\Omega(x),~z\in S_2(y)\}$.
It follows from \eqref{chain2} that
$$
\begin{aligned}
&\limsup\limits_{(x,y,z)\stackrel{C}{\longrightarrow}(x_k,y_k,z_k)}
\frac{\langle(x_k^*,0,-z_k^*),(x,y,z)-(x_k,y_k,z_k)\rangle}
{\|(x,y,z)-(x_k,y_k,z_k)\|_{X\times Y\times Z}}\\
=&\limsup\limits_{(x,y,z)\stackrel{C}{\longrightarrow}(x_k,y_k,z_k)}
\frac{\langle(x_k^*,-z_k^*),(x,z)-(x_k,z_k)\rangle}
{\sqrt{\|x-x_k\|_X^2+\|y-y_k\|_Y^2+\|z-z_k\|_Z^2}}\\
\leq&\limsup\limits_{(x,z)\stackrel{\operatorname{gph} S|_\Omega}{\longrightarrow}(x_k,z_k)}
\frac{\langle(x_k^*,-z_k^*),(x,z)-(x_k,z_k)\rangle}
{\sqrt{\|x-x_k\|_X^2+\|z-z_k\|_Z^2}}\\
&\leq\varepsilon_k\|(x_k^*,0,-z_k^*)\|_{X^*\times Y^*\times Z^*},
\end{aligned}
$$
which readily verifies the inclusion
\begin{equation*}
(x_k^*,0,-z_k^*)\in\widehat N^c_{\varepsilon_k}\big((x_k,y_k,z_k);C\big),\quad \mbox{for~all}~ k\in\mathbb N.
\end{equation*}
Let $Q:=\operatorname{gph} S_1|_\Omega\times\operatorname{gph} S_2$. Defining the single-valued continuous mapping $h\colon X\times Y\times Z\to X\times\times Y\times Y\times Z$ by $h(x,y,z):=(x,y,y,z)$, we have $C=h^{-1}(Q)$. Denote further
\begin{equation}\label{chain8}
\begin{aligned}
&\Omega_1:=X\times Y\times Z\times Q,~~\Omega_2:=\operatorname{gph}h,\;\mbox{ and}\\
&w_k:=(x_k,y_k,z_k,x_k,y_k,y_k,z_k),\;\;w_k^*:=(x_k^*,0,-z_k^*,0,0,0,0).
\end{aligned}
\end{equation}
Applying Lemma~\ref{F-1} to $h$ and $Q$ tells us that
\begin{equation}\label{chain9}
w_k^*\in\widehat N_{\varepsilon_k}^c(w_k;\Omega_1\cap\Omega_2),\quad k\in\mathbb N.
\end{equation}
It follows from (\ref{chain2}) and Proposition~\ref{JT} that there exists $\delta_k>0$ such that
\begin{equation}\label{chain10}
x_k^*\in J\big(T(x;\Omega)\big)+\varepsilon_k B_{X^*}\;\mbox{ for all }\;x\in \Omega\cap(x_k+\delta_k B_X).
\end{equation}
With the notation $\xi_k:={\rm min}\{\delta_k,\varepsilon_k\}$,
$W:=X\times Y\times Z\times X\times Y\times Y\times Z$, and
$W^*:=X^*\times Y^*\times Z^*\times X^*\times Y^*\times Y^*\times Z^*$, it is easy to see that there exists $\ell>0$ such that
\begin{equation}\label{chain12}
\|w_k^*\|_{W^*}\le\ell\;\mbox{ for all }\;k\in\mathbb N.
\end{equation}
By \eqref{chain9}, it follows from \cite[Proposition~3.3(i)]{Mordukhovich2023} whose proof is based on the approximate extremal principle (see \cite[Lemma~3.1]{Mordukhovich2006}), that there exist sequences
\begin{equation}\label{chain13}
\hat w_k\in\Omega_1\cap(w_k+\xi_k B_W),~~\tilde w_k\in\Omega_2\cap(w_k+\xi_k B_W),
\end{equation}
\begin{equation}\label{chain14}
\hat w_k^*\in\widehat N(\hat w_k;\Omega_1),\;\mbox{ and }\;
\tilde w_k^*\in\widehat N(\hat w_k;\Omega_2)
\end{equation}
such that for each $k\in\mathbb N$ we have the estimates
\begin{equation}\label{chain15}
\begin{aligned}
&\|\lambda_k w_k^*-\hat w_k^*-\tilde w_k^*\|_{W^*}\leq2(1+\ell)\varepsilon_k~~\operatorname{and}\\
&1-(1+\ell)\varepsilon_k\leq\max\big\{\lambda_k,\|\hat w_k^*\|_{W^*}\big\}\leq1+(1+\ell)\varepsilon_k.
\end{aligned}
\end{equation}
The structures in \eqref{chain8} and \cite[Lemma~2.6]{Mordukhovich2023} lead us to the representations
\begin{equation}\label{chain18}
\begin{array}{ll}
\hat w_k=(\hat x_{1k},\hat y_{1k},\hat z_{1k},\hat x_{2k},\hat y_{2k},\hat y_{3k},\hat z_{2k}),\;\tilde w_k=(\tilde x_k,\tilde y_k,\tilde z_k,\tilde x_k,\tilde y_k,\tilde y_k,\tilde z_k),\\
\hat w_k^*=(0,0,0,\hat x_{k}^*,\hat y_{1k}^*,\hat y_{2k}^*,\hat z_k^*),\;
\tilde w_k^*=(-\tilde x_k^*,-\tilde y_{1k}^*-\tilde y_{2k}^*,-\tilde z_k^*,\tilde x_k^*,\tilde y_{1k}^*,\tilde y_{2k}^*,\tilde z_k^*).
\end{array}
\end{equation}
Since $w_k\to(\bar x,\bar y,\bar z,\bar x,\bar y,\bar y,\bar z)$ and $\xi_k\to0$,
we get
\begin{equation}\label{chain20}
(\hat x_{2k},\hat y_{2k},\hat y_{3k},\hat z_{2k})\to(\bar x,\bar y,\bar y,\bar z)\;\mbox{ as }\;k\to\infty.
\end{equation}
It follows now from \eqref{chain8}, \eqref{chain14}, \eqref{chain18}, and \cite[Proposition~1.2]{Mordukhovich2006} that
$$
(\hat x_{k}^*,\hat y_{1k}^*,\hat y_{2k}^*,\hat z_k^*)\in\widehat N\big((\hat x_{2k},\hat y_{2k},\hat y_{3k},\hat z_{2k});Q\big),
$$
which, being combined with the definition of $Q$, yields
\begin{equation}\label{chain21}
\begin{array}{ll}
(\hat x_{k}^*,\hat y_{1k}^*)\in\widehat N\big((\hat x_{2k},\hat y_{2k});\operatorname{gph} S_1|_\Omega\big),\\
(\hat y_{2k}^*,\hat z_{k}^*)\in\widehat N\big((\hat y_{3k},\hat z_{2k});\operatorname{gph} S_2\big),\quad k\in\mathbb N.
\end{array}
\end{equation}
In particular, we arrive at the inclusion
\begin{equation}\label{chain22}
(\hat y_{1k}^*,0)\in[\widehat N\big((\hat y_{2k},\hat x_{2k});\operatorname{gph} S_1|_\Omega^{-1}\big)+{\|\hat x_k^*\|_{X^*}}B_{X^*\times Y^*}]\cap\big[Y^*\times J\big(T(\hat x_{2k};\Omega)\big)\big].
\end{equation}

Next we claim that there exists a positive number $\lambda_0$ such that $\lambda_k\geq\lambda_0$ for all $k$. Indeed, supposing the contrary gives us without loss of generality that $\lambda_k\downarrow0$ as $k\to\infty$. Then it follows from \eqref{chain12} and \eqref{chain15} that
\begin{equation}\label{chain23}
\hat w_k^*+\tilde w_k^*\to0,~~\|\hat w_k^*\|_{W^*}\to1,~~\operatorname{and}~~\|\tilde w_k^*\|_{W^*}\to1.
\end{equation}
Employing further \eqref{chain18} and \eqref{chain23} shows that
\begin{equation}\label{chain24}
\hat x_k^*\to0,~\hat y_{1k}^*+\hat y_{2k}^*\to0,~\hat z_k^*\to0,~~\operatorname{and}~~(\|\hat y_{1k}^*\|_{Y^*},\|\hat y_{2k}^*\|_{Y^*})\to(1/2,1/2).
\end{equation}
This tells us therefore that
\begin{equation}\label{chain25}
\|\hat y_{1k}^*\|_{Y^*},\|\hat y_{2k}^*\|_{Y^*}\geq\frac{1}{4}~~\mbox{for sufficiently large }\;k,
\end{equation}
and that there exists $y^*\in Y^*$ for which
\begin{equation}\label{chain100}
\hat y_{1k}^*\rightharpoonup^* y^*,~~\hat y_{2k}^*\rightharpoonup^*-y^*\;\mbox{ as }\;k\to\infty.
\end{equation}
It follows from \eqref{chain22} and \eqref{chain25} that
\begin{equation*}
(\hat y_{1k}^*,0)\in\widehat N_{4\|\hat x_k^*\|_{X^*}}^c\big((\hat y_{2k},\hat x_{2k});\operatorname{gph} S_1|_\Omega^{-1}\big)\cap\big[Y^*\times J\big(T(\hat x_{2k};\Omega)\big)\big].
\end{equation*}
Combining the latter with \eqref{chain20}, \eqref{chain21}, \eqref{chain24}, and \eqref{chain100} brings us to
$$
y^*\in D^*S_2(\bar y|\bar z)(0)\cap(-D_\Omega^{*-1}S_1^{-1}(\bar y|\bar x)(0)),
$$
which ensures by the qualification condition \eqref{mixedQC1} that $y^*=0$. Moreover, the imposed requirements on either $S_2$ is PSNC at $(\bar y,\bar z)$, or $S_1^{-1}$ is completely mirror PSNC relative to $\Omega$ at $(\bar y,\bar x)$ tell us that
$$
\mbox{either }\;\hat y_{1k}^*\to0,\;\mbox{ or }\;\hat y_{2k}^*\to 0\;\mbox{ as }\;k\to\infty,
$$
respectively. This contradicts \eqref{chain25} and hence verifies the claimed existence of $\lambda_0>0$.

Thus we can assume without loss of generality that $\lambda_k=1$ for all $k\in\mathbb N$ and then deduce from \eqref{chain8}, \eqref{chain15}, and \eqref{chain18} that
$$
\begin{aligned}
&\|(\hat x_k^*,\hat y_{1k}^*,\hat y_{2k}^*,\hat z_k^*)+(\tilde x_k^*,\tilde y_{1k}^*,\tilde y_{2k}^*,\tilde z_k^*)\|_{X^*\times Y^*\times Y^*\times Z^*}\\
&+\|(x_k^*,0,-z_k^*)-(-\tilde x_k^*,-\tilde y_{1k}^*-\tilde y_{2k}^*,-\tilde z_k^*)\|_{X^*\times Y^*\times Z^*}\leq2(1+\ell)\varepsilon_k,
\end{aligned}
$$
which brings us to the estimate
\begin{equation}\label{chain30}
\|x_k^*-\hat x_k^*\|_{X^*}+\|z_k^*+\hat z_k^*\|_{Z^*}+\|\hat y_{1k}^*+\hat y_{2k}^*\|_{Y^*}\leq2(1+\ell)\varepsilon_k
\end{equation}
and ensures together with \eqref{chain21} and \eqref{chain30} that
$$
\begin{aligned}
&(x_k^*,-\hat y_{2k}^*)\in\widehat N\big((\hat x_{2k},\hat y_{2k});{\rm gph}\,S_1|_\Omega\big)+2(1+\ell)\varepsilon_kB_{X^*\times Y^*},\\
&(\hat y_{2k}^*,-z_k^*)\in\widehat N\big((\hat y_{3k},\hat z_{2k});{\rm gph}\,S_2\big)+2(1+\ell)\varepsilon_kB_{Y^*\times Z^*}.
\end{aligned}
$$
Furthermore, we deduce from \eqref{chain2}, \eqref{chain10}, and \eqref{chain13} that there exists $x_{2k}^*\in J\big(T(\hat x_{2k};\Omega)\big)$ satisfying $\|x_{2k}^*-x_k^*\|_{X^*}\le\varepsilon_k$, and hence
\begin{equation}\label{chain31}
\begin{array}{ll}
(x_{2k}^*,-\hat y_{2k}^*)\in&[\widehat N\big((\hat x_{2k},\hat y_{2k});\operatorname{gph} S_1|_\Omega\big)+{2(2+\ell)\varepsilon_k}B_{X^*\times Y^*}]
\\&\cap\big[J\big(T(\hat x_{2k};\Omega)\big)\times Y^*\big].
\end{array}
\end{equation}
It follows from \eqref{chain15} and \eqref{chain18} that the sequence $\{\hat y_{2k}^*\}$ is bounded. Hence selecting a subsequence if necessary gives us $y^*\in Y^*$ such that $\hat y_{2k}^*\rightharpoonup^*y^*$ as $k\to\infty$. Thus we just need to show that $x^*\in\widetilde D_\Omega^{c*}S_1(\bar x|\bar y)(y^*)$ and $y^*\in D^*S_2(\bar y|\bar z)(z^*)$. If there exists $\alpha>0$ such that $\|(x_{2k}^*,-\hat y_{2k}^*)\|_{X^*\times Y^*}>\alpha$ for all $k$, then deduce from \eqref{chain31} that
$$
(x_{2k}^*,-\hat y_{2k}^*)\in\widehat N_{2(2+\ell)\varepsilon_k/\alpha}^c\big((\hat x_{2k},\hat y_{2k});\operatorname{gph} S_1|_\Omega\big)\cap\big[J\big(T(\hat x_{2k};\Omega)\big)\times Y^*\big],
$$
which implies that $x^*\in\widetilde D_\Omega^{c*}S_1(\bar x|\bar y)(y^*)$. If $(x_{2k}^*,-\hat y_{2k}^*)\to (0,0)$ as $k\to\infty$, then $x^*=0$, $y^*=0$, and the claim is trivial. Similarly, we also get that $y^*\in D^*S_2(\bar y|\bar z)(z^*)$ and thus complete the proof of assertion (i).\vspace*{0.02in}

Next we verify assertion (ii). Observe that the strict differentiability of $S_2=f$ ensures that the qualification condition \eqref{mixedQC1} and the PSNC property of $S_2$ are satisfied automatically. To proceed with the proof of \eqref{chainrule5}, pick any $x^*\in D_\Omega^{c*}S(\bar x|\bar z)(z^*)$ and find
\begin{equation*}
\varepsilon_k\downarrow0,~~ (x_k,z_k)\stackrel{\operatorname{gph} S|_\Omega}{\longrightarrow}(\bar x,\bar z),\quad\mbox{and}
\end{equation*}
\begin{equation*}
(x_k^*,-z_k^*)\in\widehat N_{\varepsilon_k}^c\big((x_k,z_k);\operatorname{gph} S|_\Omega\big)\cap[J(T(x_k;\Omega)\big)\times Z^*]
\end{equation*}
with the convergence of the dual sequences
\begin{equation*}
z_k^*\to z^*~~\operatorname{and}~~x_k^*\rightharpoonup^*x^*\;\mbox{ as }\;k\to\infty.
\end{equation*}
Since $G$ from \eqref{G-semi} is inner semicontinuous relative to $\Omega\times Z$ at $(\bar x,\bar z,\bar y)$ with $\bar y\in G(\bar x,\bar z)$, there exist elements $y_k\in G(x_k,z_k)$ such that $y_k\to \bar y$ as $k\to\infty$. Taking $S_2=f$ in the proof of assertion (i), we find
\begin{equation}\label{chainiii1}
(\hat x_k,\hat y_k)\in (x_k,y_k)+\xi_k B_{X\times Y},~~(\tilde y_k,\tilde z_k)\in (y_k,z_k)+\xi_k B_{Y\times Z},
\end{equation}
and $(\hat x_k^*,\hat y_k^*)\in (x_k^*+\varepsilon_k B_{X^*})\times Y^*$ such that
\begin{equation}\label{chainiii2}
\begin{array}{ll}
(\hat x_k^*,-\hat y_k^*)\in [\widehat N\big((\hat x_k,\hat y_k);\operatorname{gph} S_1|_\Omega\big)+{(2+\ell)\varepsilon_k}B_{X^*\times Y^*}]\cap\big[J\big(T(\hat x_k;\Omega)\big)\times Y^*\big],\\
(\hat y_k^*,-z_k^*)\in\widehat N\big((\tilde y_k,\tilde z_k);\operatorname{gph} f\big)+(1+\ell)\varepsilon_k(B_{Y^*}\times B_{Z^*}).
\end{array}
\end{equation}
We clearly have the convergence $\hat x_k^*\rightharpoonup^*x^*$. Utilizing this together with \eqref{chainiii1} gives us
\begin{equation*}
(\hat x_k,\hat y_k)\to(\bar x,\bar y)~~\operatorname{and}~~(\tilde y_k,\tilde z_k)\to(\bar y,\bar z)\;\mbox{ as }\;k\to\infty.
\end{equation*}
It follows from \eqref{chainiii2} that there exist pairs
\begin{equation}\label{chainiii5}
\begin{array}{ll}
(\tilde y_k^*,-\tilde z_k^*)\in (\hat y_k^*,-z_k^*)+(1+\ell)\varepsilon_kB_{Y^*\times Z^*}\;\mbox{ with}\\
(\hat y_k^*,-\hat z_k^*)\in\widehat N\big((\tilde y_k,\tilde z_k);\operatorname{gph}f\big),\quad k\in\mathbb N.
\end{array}
\end{equation}
Thus \cite[Lemma~2.6]{Mordukhovich2023} gives us a sequence $\rho_k\downarrow0$ ensuring that the estimate
\begin{equation}\label{chainiii7}
\|\tilde y_k^*-\nabla f(\bar y)^*\tilde z_k^*\|_{Y^*}\leq\rho_k
\end{equation}
holds for sufficiently large $k$. It follows from   \eqref{chainiii5} and \eqref{chainiii7} that
\begin{equation*}
\tilde z_k^*\to z^*~~\operatorname{and}~\tilde y_k^*\to\nabla f(\bar y)z^*\;\mbox{ as }\;k\to\infty.
\end{equation*}
Passing to the limit in \eqref{chainiii2} with using the established convergence brings us to
$$
x^*\in D_\Omega^{c*} S_1(\bar x|\bar y)(\nabla f(\bar y)^*z^*),
$$
which verifies \eqref{chainrule5} and thus completes the proof of the theorem.
\end{proof}

Observe that assertions (i) and (ii) of Theorem~\ref{chainrule} are mutually independent. Indeed, inclusion \eqref{chainrule5} for the mixed contingent coderivative of the composition in (ii) is expressed via the {\it mixed} one for $S_1$ while not via its {\it normal} counterpart as in \eqref{chainrule1}. \vspace*{0.05in}

To conclude this section, we recall from \cite{Mordukhovich2023} the notions of {\it relative metric regularity} and {\it relative linear openness} of multifunctions.

\begin{definition}\label{metreg} Let $S\colon X\tto Y$ be set-valued mapping between Banach spaces, let $\O\subset X$, and let $(\ox,\oy)\in\gph S|_\O$. Then we say that:

{\bf(i)} $S$ is is {\sc metrically regular relative to} $\Omega$ around $(\bar x,\oy)$ if there exist neighborhoods $U$ of $\ox$ and $V$ of $\oy$ together with a number $\kappa\ge 0$ such that
$$
d(x';S^{-1}(y))\le\kappa d(y;S(x'))\;\mbox{ for all }\;x'\in \Omega\cap U\;\mbox{ and }\;y\in S(\O)\cap V.
$$

{\bf(ii)} $S$ is {\sc linearly open relative to} $\Omega$ around $(\bar x,\bar y)$ if there exist neighborhoods $U$ of $\ox$ and $V$ of $\oy$ together with a number $\mu\ge 0$ such that
\begin{equation*}
\big(S(x')+\ve B_Y\big)\cap S(\O)\cap V\subset S\big(x'+\mu\ve B_X)\;\mbox{ for all }\;x'\in\Omega\cap U.
\end{equation*}    
\end{definition}

It is known that the relative metric regularity and relative linear openness of $S$ around $(\ox,\oy)$ are equivalent to each other (with different moduli) being also equivalent to the relative Lipschitz-like property of the inverse $S^{-1}$ around $(\oy,\ox)$. Thus the properties from Definition~\ref{metreg} admit the corresponding coderivative characterizations; cf.\
\cite[Theorem~3.9]{Mordukhovich2023}. This allows us to obtain the following consequence of Theorems~\ref{chainrule} and \ref{criterion}.

\begin{corollary}
In the general setting of Theorem~{\rm\ref{chainrule}}, suppose that either $S_1|_\Omega$ is metrically regular/linearly open around $(\bar y,\bar x)$, or $S_2$ is Lipschitz-like around $(\bar y,\bar z)$. Then both conic contingent coderivative chain rules in \eqref{chainrule1} and \eqref{chainrule2} hold.
\end{corollary}
\vspace*{-0.2in}

\section{Relative Extremal Principle and Sum Rules for Conic Contingent Coderivatives}\label{sec:sum}

The primary goal of this section is to establish, for the first time in the literature, {\it pointbased sum rules} for both normal and  mixed conic contingent coderivatives from Definition~\ref{coderivatives}(ii,iii), and hence for their equivalent versions from \cite{Mordukhovich2023}. To derive such rules, we develop two major results of the undoubted own interest for variational analysis and applications, which go far beyond coderivative calculus. They are the {\it relative extremal principle} for systems of closed sets and the {\it relative fuzzy intersection rule} for $\ve$-regular normals whose proof is based on the relative extremal principle.\vspace*{0.02in}

We start with the definition of {\it relative extremal points} of set systems, where $\Pi_X(\Lambda)$ signifies the projection of a set $\Lambda\subset X\times Y$ onto the space $X$.

\begin{definition}\label{ext-point} Let $\Lambda_1,\Lambda_2\subset X\times Y$ and $\O\subset X$ be nonempty subsets of Banach spaces. We say that $(\bar x,\bar y)\in\Lambda_1\times\Lambda_2$ is a {\sc local extremal point} of the set system $\{\Lambda_1,\Lambda_2\}$ relative to $\Omega$ if there exist sequence $\{b_k\}\subset X$ and a neighborhood $U$ of $(\bar x,\bar y)$ such that $b_k\to 0$ as $k\to\infty$, $\Pi_X(\Lambda_1\cup\Lambda_2)\subset \Omega$,  and
$$
\big(\Lambda_1+(0,b_k)\big)\cap\Lambda_2\cap U=\emp\;\mbox{ for all~large }\;k\in\mathbb N.
$$
\end{definition}

Here is the {\it relative extremal principle} for closed sets in reflexive Banach spaces.

\begin{theorem}\label{extremal}
Let $X$ and $Y$ be two reflexive Banach spaces, let $\Omega$ be a closed and convex subset of $X$, and let $(\bar x,\bar y)$ be a local extremal point of the system of closed sets $\{\Lambda_1,\Lambda_2\}$ relative to $\Omega$. Then for any $\varepsilon>0$, there exist pairs $(x_i,y_i)\in \Lambda_i\cap[(\bar x+\varepsilon B_X)\times(\bar y+\varepsilon B_Y)]$ and $(x_i^*,y_i^*)\in X^*\times Y^*$, $i=1,2$, such that we have the relationships
\begin{equation}\label{extr0}
\begin{aligned}
&(x_i^*,y_i^*)\in\big[\widehat N\big((x_i,y_i);\Lambda_i\big)+\varepsilon (B_{X^*}\times B_{Y^*})\big]\cap \big[J\big(T(x_i;\Omega)\big)\times Y^*\big],\quad i=1,2,\\
&(x_1^*,y_1^*)+(x_2^*,y_2^*)=0,\;\mbox{ and }\;\|(x_1^*,y_1^*)\|_{X^*\times Y^*}+\|(x_2^*,y_2^*)\|_{X^*\times Y^*}=1.
\end{aligned}
\end{equation}
\end{theorem}
\begin{proof} Take a neighborhood $U$ of $(\bar x,\bar y)$ from Definition~\ref{ext-point} such that for any $\varepsilon>0$ there exists $b\in Y$ with $\|b\|_Y\leq\varepsilon^3/2$ and $(\Lambda_1+(0,b))\cap\Lambda_2\cap U=\emptyset$. Assume for simplicity that $U=X\times Y$ and that $\varepsilon<1/2$. Then considering the function
$$
\varphi(z):=\|(x_1,y_1)-(x_2,y_2)+(0,b)\|_{X\times Y}~~\operatorname{for}~z=(x_1,y_1,x_2,y_2)\in X\times Y\times X\times Y,
$$
we conclude that $\varphi(z)>0$ on $\Lambda_1\times\Lambda_2$, and hence $\varphi$ is Fr$\acute{\text{e}}$chet differentiable at any point $z\in\Lambda_1\times\Lambda_2$ due to our choice of the equivalent norm on the reflexive space $X\times Y$; see Section~\ref{sec:pre}. In what follows, we always use the product norm 
$$
\|z\|:=(\|x_1\|_X^2+\|y_1\|_Y^2+\|x_2\|_X^2+\|y_2\|_Y^2)^{1/2}\;\mbox{ on }\;X\times Y\times X\times Y,
$$
which is obviously Fr$\acute{\text{e}}$chet differentiable at nonzero points of $X\times Y\times X\times Y$.

Let $z_0:=(\bar x,\bar y,\bar x,\bar y)$ and observe that the set
$$
W(z_0):=\{z\in\Lambda_1\times\Lambda_2~|~
\varphi(z)+\varepsilon\|z-z_0\|^2/2\leq\varphi(z_0)\}
$$
is nonempty and closed. Furthermore, for each $z:=(x_1,y_1,x_2,y_2)\in W(z_0)$ we get
\begin{equation}\label{extr}
\|x_1-\bar x\|_X^2+\|y_1-\bar y\|_Y^2+\|x_2-\bar x\|_X^2+\|y_2-\bar y\|_Y^2\leq2\varphi(z_0)/\varepsilon=2\|b\|_Y/\varepsilon
\leq\varepsilon^2,
\end{equation}
which implies that $W(z_0)\subset (\bar x+\varepsilon B_X)\times(\bar y+\varepsilon B_Y)\times(\bar x+\varepsilon B_X)\times(\bar y+\varepsilon B_Y)$. Next we inductively construct sequences of vectors $z_k\in\Lambda_1\times\Lambda_2$ and nonempty closed sets $W(z_k)$, $k\in\mathbb N$, as follows. Given $z_k$ and $W(z_k)$, $k=0,1,\ldots$, select $z_{k+1}\in W(z_k)$ satisfying
$$
\varphi(z_{k+1})+\varepsilon\sum\limits_{j=0}^k\frac{\|z_{k+1}-z_j\|^2}{2^{j+1}}
<\inf\limits_{z\in W(z_k)}\left\{\varphi(z)+\varepsilon\sum\limits_{j=0}^k
\frac{\|z-z_j\|^2}{2^{j+1}}\right\}+\frac{\varepsilon^3}{2^{3k+2}}.
$$
Construct further the set $W(z_{k+1})$ by
$$
W(z_{k+1}):=\left\{z\in\Lambda_1\times\Lambda_2~|~\varphi(z)
+\varepsilon\sum\limits_{j=0}^{k+1}\frac{\|z-z_j\|^2}{2^{j+1}}
\leq\varphi(z_{k+1})+\varepsilon\sum\limits_{j=0}^k
\frac{\|z_{k+1}-z_j\|^2}{2^{j+1}}\right\}
$$
and then define the function
\begin{equation}\label{phi}
\phi(z):=\varphi(z)+\varepsilon\sum\limits_{j=0}^\infty\frac{\|z-z_j\|^2}{2^{j+1}}.
\end{equation}
Following the proof of \cite[Theorem~2.10]{Mordukhovich2006}, we know that there exists $\bar z:=(\bar x_1,\bar y_1,\bar x_2,\bar y_2)\in X\times Y\times X\times Y$ such that $z_k\to \bar z$ as $k\to\infty$ and that $\bar z$ is a minimum point of $\phi$ over the set $\Lambda_1\times \Lambda_2$. Hence the function $\psi(z):=\phi(z)+\dd_{\Lambda_1\times\Lambda_2}(z)$ achieves at $\bar z$ its minimum over $X\times Y\times X\times Y$, which immediately yields $0\in\hat\partial \psi(\bar z)$. Note that $\phi$ is Fr$\acute{\text{e}}$chet differentiable at $\bar z$ due to $\varphi(\bar z)\not=0$ and the smoothness of $\|\cdot\|^2$. Now applying the sun rule of \cite[Proposition~1.107(i)]{Mordukhovich2006}, we get
$$
-\nabla \phi(\bar z)\in\widehat N(\bar z;\Lambda_1\times\Lambda_2)=\widehat N\big((\bar x_1,\bar y_1);\Lambda_1\big)\times \widehat N\big((\bar x_2,\bar y_2);\Lambda_2\big).
$$
It follows from \eqref{phi} that $\nabla \phi(\bar z)=(u_1^*,v_1^*,u_2^*,v_2^*)\in X^*\times Y^*\times X^*\times Y^*$, where
$$
\begin{aligned}
&(u_1^*,v_1^*):=(x^*,y^*)+\varepsilon\sum\limits_{j=0}^\infty\frac{\big(J(\bar x_1-x_{1j}),J(\bar y_1-y_{1j})\big)}{2^j},\\
&(u_2^*,v_2^*):=-(x^*,y^*)+\varepsilon\sum\limits_{j=0}^\infty\frac{\big(J(\bar x_2-x_{2j}),J(\bar y_2-y_{2j})\big)}{2^j}
\end{aligned}
$$
with $(x_{1j},y_{1j},x_{2j},y_{2j})=z_j$, $j=0,1,\ldots$, and where
$$
(x^*,y^*):=\left(\frac{J(\bar x_1-\bar x_2)}{\sqrt{\|\bar x_1-\bar x_2\|_X^2+\|\bar y_1-\bar y_2+b\|_Y^2}},\frac{J(\bar y_1-\bar y_2+b)}{\sqrt{\|\bar x_1-\bar x_2\|_X^2+\|\bar y_1-\bar y_2+b\|_Y^2}}\right).
$$
The latter tells us that $\|(x^*,y^*)\|_{X^*\times Y^*}=1$. It follows from \eqref{extr} that
$$
\begin{aligned}
&\big\|\big(J(\bar x_i-x_{ij}),J(\bar y_i-y_{ij})\big)\big\|_{X^*\times Y^*}
=\sqrt{\|\bar x_i-x_{ij}\|_X^2+\|\bar y_i-y_{ij}\|_Y^2}\\
\leq&\sqrt{\|\bar x_i-\bar x\|_X^2+\|\bar y_i-\bar y\|_Y^2}+\sqrt{\|\bar x-x_{ij}\|_X^2+\|\bar y-y_{ij}\|_Y^2}\le 2\varepsilon<1
\end{aligned}
$$
for $j=0,1,\ldots$ and $i=1,2$. Therefore,
$$
\sum\limits_{j=0}^\infty\frac{\big\|\big(J(\bar x_1-x_{1j}),J(\bar y_1-y_{1,j})\big)\big\|_{X^*\times Y^*}}{2^j}\leq2,\quad i=1,2.
$$
Putting $(x_i,y_i):=(\bar x_i,\bar y_i)$ and $(x_i^*,y_i^*):=(-1)^i(x^*,y^*)/2$ for $i=1,2$, we arrive at the relationships in \eqref{extr0} by $J(\bar x_2-\bar x_1)\in J(T(\bar x_1;\Omega))$ and $J(\bar x_1-\bar x_2)\in J(T(\bar x_2;\Omega))$ due to the convexity of $\Omega$ and the construction of $(x^*,y^*)$, which completes the proof.
\end{proof}

The next theorem establishes a {\it relative fuzzy intersection rule} for closed sets in Hilbert spaces. The proof is based on the relative extremal principle from Theorem~\ref{extremal} while using the convexity-preserving property of the duality mapping, which holds in Hilbert spaces; see Remark~\ref{rem-conv} for more discussions. Note that in Hilbert spaces, the duality mapping reduce to the identity operator.

\begin{theorem}\label{fuzzy} Let $X$ and $Y$ be Hilbert spaces, let $\Theta_1,\Th_2\subset X\times Y$ be locally closed around $(\bar x,\bar y)\in\Theta_1\cap\Theta_2$, and let $\Pi_X(\Theta_1\cup\Theta_2)\subset\Omega$ for some closed and convex set $\Omega\subset X$. Fix $\varepsilon\in(0,1)$ and pick $(x^*,y^*)\in\widehat N_\varepsilon^c\big((\bar x,\bar y);\Theta_1\cap\Theta_2\big)\cap[T(\bar x;\Omega)\times Y]$. Then for any $\nu>0$, there exist $\lambda\ge 0$,  $(x_i,y_i)\in\Theta_i\cap[(\bar x+\nu B_X)\times(\bar y+\nu B_Y)]$, and $(x_i^*,y_i^*)\in[\widehat N\big((x_i,y_i);\Theta\big)+(\|(x^*,y^*)\|_{X\times Y}+\nu)\varepsilon(B_X\times B_Y)]\cap[T(x_i;\Omega)\times Y]$, $i=1,2$, such that
\begin{equation}\label{fuzzy0}
\lambda (x^*,y^*)-(x_1^*,y_1^*)-(x_2^*,y_2^*)\in\nu(B_X\times B_Y)\;\mbox{ and }\;\max\{\lambda,\|(x_1^*,y_1^*)\|_{X\times Y}\}=1.
\end{equation}
\end{theorem}
\begin{proof}
Take any $\nu>0$ and $(x^*,y^*)\in\widehat N_\varepsilon^c \big((\bar x,\bar y);\Theta_1\cap\Theta_2\big)\cap[T(\bar x;\Omega)\times Y]$, and denote $L:=\|(x^*,y^*)\|_{X\times Y}$. Then there exists a neighborhood $U$ of $(\bar x,\bar y)$ such that
\begin{equation}\label{fuzzy1}
\langle x^*,x-\bar x\rangle+\langle y^*,y-\bar y\rangle-(L+\frac{1}{2}\nu)\varepsilon\|(x,y)-(\bar x,\bar y)\|_{X\times Y}\le 0
\end{equation}
for all $(x,y)\in\Theta_1\cap\Theta_2\cap U$. Define subsets of $X\times Y\times \mathbb R$ by
\begin{align*}
\Lambda_1:=\big\{(x,y,\alpha)\in X\times Y\times \mathbb R~\big|&~(x,y)\in\Theta_1,~\alpha\ge 0\big\};\\
\Lambda_2:=\big\{(x,y,\alpha)\in X\times Y\times \mathbb R~\big|&~(x,y)\in\Theta_2,~\alpha\leq \langle x^*,x-\bar x\rangle+\langle y^*,y-\bar y\rangle\\
&-(L+\frac{1}{2}\nu)\varepsilon\|(x,y)-(\bar x,\bar y)\|_{X\times Y}\big\}.
\end{align*}
Observe that $(\bar x,\bar y,0)\in\Lambda_1\cap\Lambda_2$ and that the sets $\Lambda_i$ are locally closed around $(\bar x,\bar y,0)$. Moreover, it is easy to check that
$$
\Lambda_1\cap\big(\Lambda_2-(0,0,\nu)\big)\cap(U\times \mathbb R)=\emptyset~~\operatorname{for~all}~\nu>0
$$
due to \eqref{fuzzy1} and the structure of $\Lambda_i$, $i=1,2$. Thus $(\bar x,\bar y,0)$ is a local extremal point relative to $\Omega$ of the set system $\{\Lambda_1,\Lambda_2\}$. For any $\tau\in(0,\frac{1}{2})$, it follows from Lemma~\ref{JT} that there exists a positive number $\rho$ such that
\begin{equation}\label{fuzzy2}
x^*\in T(x;\Omega)+\tau\varepsilon B_X\;\mbox{ for all }\;x\in \Omega\cap(\bar x+2\rho B_X).
\end{equation}
Applying Theorem~\ref{extremal} to the set system
$\{\Lambda_1,\Lambda_2\}$, we find $(x_i,y_i,\alpha_i)\in\Lambda_i$ and $(x_i^*,y_i^*,\lambda_i)\in\widehat N\big((x_i,y_i,\alpha_i);\Lambda_i\big)\cap[T(x_i;\Omega)\times Y\times \mathbb R+\tau\varepsilon(B_X\times B_Y\times \mathbb R)]$, $i=1,2$, such that
\begin{equation}\label{fuzzy3}
\begin{aligned}
&\|x_1^*+x_2^*\|_X^2+\|y_1^*+y_2^*\|_Y^2+|\lambda_1+\lambda_2|^2<\tau^2\varepsilon^2,\\
&\sqrt{\|x_i^*\|_X^2+\|y_i^*\|_Y^2+|\lambda_i|^2}\in
\left(\frac{1}{2}-\tau\varepsilon,\frac{1}{2}+\tau\varepsilon\right),\\
&\|x_i-\bar x\|_X^2+\|y_i-\bar y\|_Y^2+|\lambda_i|^2<{\rm min}\{\tau^2\varepsilon^2,\rho^2\},\quad i=1,2.
\end{aligned}
\end{equation}
It is easy to deduce from the above that $\lambda_1\leq0$, $(x_1^*,y_1^*)\in\widehat N\big((x_1,y_1);\Theta_1\big)\cap[T(x_1;\Omega)\times Y+\tau\varepsilon (B_X\times B_Y)]$, and
\begin{equation}\label{fuzzy4}
\limsup\limits_{(x,y,\alpha)\stackrel{\Lambda_2}{\longrightarrow}(x_2,y_2,\alpha_2)}
\frac{\langle x_2^*,x-x_2\rangle+\langle y_2^*,y-y_2\rangle+\lambda_2(\alpha-\alpha_2)}
{\sqrt{\|x-x_2\|_X^2+\|y-y_2\|_Y^2+|\alpha-\alpha_2|^2}}\leq0.
\end{equation}
We get from the structure of $\Lambda_2$ that $\lambda_2\geq0$ and
\begin{equation}\label{fuzzy5}
\alpha_2\leq\langle x_2^*,x_2-\bar x\rangle+\langle y_2^*,y_2-\bar y\rangle -(L+\frac{1}{2}\nu)\varepsilon\|(x_2,y_2)-(\bar x,\bar y)\|_{X\times Y}.
\end{equation}
If inequality \eqref{fuzzy5} is strict, then \eqref{fuzzy4} yields $\lambda_2=0$ and $(x_2^*,y_2^*)\in\widehat N\big((x_2,y_2);\Theta_2\big)\cap[T(x_2;\Omega)\times Y+\tau\varepsilon (B_X\times B_Y)]$. In this case, we arrive by using \eqref{fuzzy3} at \eqref{fuzzy0} with $\lambda=0$.

It remains to consider the case of equality in \eqref{fuzzy5}. This gives us  $(x,y,\alpha)\in\Lambda_2$ with
$$
\alpha=\langle x^*,x-\bar x\rangle+\langle y^*,y-\bar y\rangle-(L+\frac{1}{2}\nu)\varepsilon\|(x,y)-(\bar x,\bar y)\|_{X\times Y},~~(x,y)\in\Theta_2\setminus\{(x_2,y_2)\}.
$$
Substituting the latter into \eqref{fuzzy4}, we find  a neighborhood $V$ of $(x_2,y_2)$ such that
\begin{equation}\label{fuzzy6}
\langle x_2^*,x-x_2\rangle+\langle y_2^*,y-y_2\rangle+\lambda_2(\alpha-\alpha_2)\leq\tau\varepsilon(\|(x,y)-(\bar x,\bar y)\|_{X\times Y}+|\alpha-\alpha_2|)
\end{equation}
for all $(x,y)\in\Theta_2\cap V$ with the corresponding  number $\alpha$ satisfying
$$
\alpha-\alpha_2=\langle x^*,x-x_2\rangle+\langle y^*,y-y_2\rangle+(L+\frac{1}{2}\nu)\varepsilon(\|(x_2,y_2)-(\bar x,\bar y)\|_{X\times Y}-\|(x,y)-(\bar x,\bar y)\|_{X\times Y}).
$$
The usage of the triangle and Cauchy–Schwarz inequalities ensures that
$$
|\alpha-\alpha_2|\leq(L+L\varepsilon+\frac{1}{2}\nu\varepsilon)\|(x,y)-(x_2,y_2)\|_{X\times Y}.
$$
Observe that the left-hand side $\vartheta$ in \eqref{fuzzy6} can be represented as
$$
\begin{aligned}
\vartheta=&\langle x_2^*+\lambda_2x^*,x-x_2\rangle+\langle y_2^*+\lambda_2y^*,y-y_2\rangle\\
&+\lambda_2(L+\frac{1}{2}\nu)\varepsilon(\|(x_2,y_2)-(\bar x,\bar y)\|_{X\times Y}-\|(x,y)-(\bar x,\bar y)\|_{X\times Y}),
\end{aligned}
$$
and thus \eqref{fuzzy6} implies the estimate
$$
\begin{aligned}
&\langle x_2^*+\lambda_2x^*,x-x_2\rangle+\langle y_2^*+\lambda_2y^*,y-y_2\rangle\\
&\leq
[\lambda_2(L+\frac{1}{2}\nu)+\tau(1+L+L\varepsilon+\nu\varepsilon)] \varepsilon
\|(x,y)-(x_2,y_2)\|_{X\times Y}
\end{aligned}
$$
for all $(x,y)\in \Theta_2\cap V$. This gives us the inclusion
\begin{equation}\label{fuzzy7}
\begin{aligned}
&(x_2^*+\lambda_2x^*,y_2^*+\lambda_2^*y^*)\in\widehat N_{\mu\varepsilon}\big((x_2,y_2);\Theta_2\big)\\
&\operatorname{with}~\mu:=\lambda_2(L+\frac{1}{2}\nu)+\tau(1+L+L\varepsilon+\nu\varepsilon).
\end{aligned}
\end{equation}
Employing Lemma~\ref{JT} ensures the existence of $\rho'\in(0,\rho)$ such that
\begin{equation}\label{fuzzy8}
x_2^*\in T(x;\Omega)+\tau\varepsilon B_X\;\mbox{ for all }\;x\in \Omega\cap(x_2+\rho'B_X).
\end{equation}
It follows from the representation of $\ve$-normals in \eqref{fuzzy7} taken from \cite[page~222]{Mordukhovich2006} that there exists a pair $(u,v)\in\Theta_2\cap[(x_2,y_2)+\rho'(B_X\times B_Y)]$ such that
$$
(x_2^*+\lambda_2x^*,y_2^*+\lambda_2y^*)\in\widehat N\big((u,v);\Theta_2\big)+(\mu+\tau)\varepsilon(B_X\times B_Y).
$$
Then we get $x_2^*\in T(u;\Omega)+\tau\varepsilon$ from \eqref{fuzzy8}, $x^*\in T(u;\Omega)+\tau\varepsilon$ from \eqref{fuzzy2} and the third inequality in \eqref{fuzzy3}.
Denoting $\eta:=\max\{\lambda_2,\|(x_2^*,y_2^*)\|_{X\times Y}\}$ gives us 
$$
\frac{1}{4}-\frac{\tau\varepsilon}{2}<\eta<\frac{1}{2}+\nu\varepsilon
$$
by \eqref{fuzzy3} with $\max\{\lambda_2,1/5\}\leq\eta<3/4$ when $\tau$ is sufficiently small. Setting
$$
\lambda:=\frac{\lambda_2}{\eta},~~(u_1^*,v_1^*):=-\frac{1}{\eta}(x_2^*,y_2^*),~~(u_2^*,v_2^*):=\frac{1}{\eta}(x_2^*+\lambda_2x^*,y_2^*+\lambda_2y^*),
$$
we have that $\lambda\geq0$, $\max\{\lambda,\|u_1^*\|_X+\|v_1^*\|_Y\}=1$, and $\lambda(x^*,y^*)=(u_1^*,v_1^*)+(u_2^*,v_2^*)$. Moreover, $(u_2^*,v_2^*)\in\widehat N\big((u,v);\Theta_2\big)+(\mu+\tau)\varepsilon/\eta(B_X\times B_Y)$ and
$$
(u_1^*,v_1^*)=\frac{1}{\eta}(x_1^*,y_1^*)-\frac{1}{\eta}(x_1^*+x_2^*,y_1^*+y_2^*)\in\widehat N\big((x_1,x_2);\Theta_1\big)+\tau\varepsilon/\eta(B_X\times B_Y)
$$
due to \eqref{fuzzy3}. This allows us to conclude that
\begin{align}
&u_1^*=\frac{x_1^*}{\eta}-\frac{x_1^*+x_2^*}{\eta}\in T(x_1;\Omega)+\frac{2\tau\varepsilon}{\eta} B_X,\label{fuzzy9}\\
&u_2^*=\frac{x_2^*}{\eta}+\frac{\lambda_2x^*}{\eta}\in T(u;\Omega)+\frac{\lambda_2+1}{\eta}\tau\varepsilon B_X,\label{fuzzy10}
\end{align}
where \eqref{fuzzy9} comes from the first inequality in \eqref{fuzzy3} by $x_1^*\in T(x_1;\Omega)+\tau\varepsilon B_X$, while \eqref{fuzzy10} follows from the fact that $x^*,x_2^*\in T(u;\Omega)+\tau\varepsilon B_X$ and the convexity of $T(u;\Omega)$. This completes the proof of the theorem by taking $\tau$ to be sufficiently small.
\end{proof}

\begin{remark}\label{rem-conv}
{\rm In the proof of Theorem~\ref{fuzzy}, we deduce the inclusion in \eqref{fuzzy10} from the convexity of $J(T(u;\Omega))=T(u;\Omega)$ in the Hilbert space $X$. Note that in reflexive spaces, the image set $J(T(u;\Omega))$ may be not convex when $\Omega$ is convex}.
\end{remark}

Now we come to the last theorem of this paper providing the {\it pointbased sum rule} for both conic contingent coderivatives from Definition~\ref{coderivatives}(ii,iii).

\begin{theorem}\label{sumrule} Let $S_i:X\rightrightarrows Y$, $i=1,2$, be set-valued mappings between Hilbert spaces, let $\Omega\subset X$ be a closed and convex set, and let $(\bar x,\bar y)\in \operatorname{gph}(S_1+S_2)|_\Omega$. Consider the mapping $G\colon X\times Y\tto Y^2$ defined by
$$
G(x,y):=\big\{(y_1,y_2)\in Y\times Y~\big|~y_1\in S_1(x),~y_2\in S_2(x),~y_1+y_2=y\big\}.
$$
Fix $(\bar y_1,\bar y_2)\in G(\bar x,\bar y)$, and let $G$ be inner semicontinuous relative to $\Omega\times Y$ at
$(\bar x,\bar y,\bar y_1,\bar y_2)$. Assume that the graphs of $S_1$ and $S_2$ are locally closed around
$(\bar x,\bar y_1)$ and $(\bar x,\bar y_2)$, respectively, that either $S_1$ is PSNC relative to $\Omega$ at $(\bar x,\bar y_1)$  or $S_2$ is PSNC relative to $\Omega$ at $(\bar x,\bar y_2)$, and that $\{S_1,S_2\}$ satisfies the mixed coderivative qualification condition
\begin{equation}\label{sumQC}
D_\Omega^{c*}S_1(\bar x|\bar y_1)(0)\cap(-D_\Omega^{c*}S_2(\bar x|\bar y_2))(0)=\{0\}.
\end{equation}
Then for all $y^*\in Y$, we have the inclusions
\begin{equation}\label{sumrule1}
D_\Omega^{c*}(S_1+S_2)(\bar x|\bar y)(y^*)\subset D_\Omega^{c*}S_1(\bar x|\bar y_1)(y^*)+D_\Omega^{c*}S_2(\bar x|\bar y_2)(y^*),
\end{equation}
\begin{equation}\label{sumrule2}
\widetilde D_\Omega^{c*}(S_1+S_2)(\bar x|\bar y)(y^*)\subset\widetilde D_\Omega^{c*}S_1(\bar x|\bar y_1)(y^*)+\widetilde D_\Omega^{c*}S_2(\bar x|\bar y_2)(y^*).
\end{equation}
\end{theorem}
\begin{proof}
We verify only \eqref{sumrule1}, since the proof of \eqref{sumrule2} is similar. Take any $(x^*,y^*)\in X\times Y$ with $x^*\in D_\Omega^{c*}(S_1+S_2)(\bar x|\bar y)(y^*)$ and find sequences
\begin{equation*}
\varepsilon_k\downarrow0,\;(x_k,y_k)\stackrel{\operatorname{gph} (S_1+S_2)|_\Omega}{\longrightarrow}(\bar x,\bar y),\;\mbox{ and}
\end{equation*}
\begin{equation*}
(x_k^*,-y_k^*)\in\widehat N_{\varepsilon_k}^c\big((x_k,y_k);\operatorname{gph} (S_1+S_2)|_\Omega\big)\cap[T(x_k;\Omega)\times Y]
\end{equation*}
satisfying the convergence conditions
$$
x_k^*\rightharpoonup^*x^*~~\operatorname{and}~~y_k^*\to y^*\;\mbox{ as }\;k\to\infty.
$$
By the inner semicontinuity of $G$ relative to $\Omega\times Y$ at $(\bar x,\bar y,\bar y_1,\bar y_2)$, there exists a sequence $(y_{1k},y_{2k})\to(\bar y_1,\bar y_2)$ with $(y_{1k},y_{2k})\in G(x_k,y_k)$ for all $k\in\mathbb N$. Define the sets
$$
\Theta_i:=\big\{(x,y_1,y_2)\in X\times Y\times Y~\big|~(x,y_i)\in\operatorname{gph} S_i|_\Omega\big\}~~\operatorname{for}~i=1,2,
$$
which are locally closed around $(\bar x,\bar y_1,\bar y_2)$ since the graphs of $S_i$ are assumed to be locally closed around $(\bar x,\bar y_i)$, $i = 1, 2$. We have $(x_k,y_{1k},y_{2k})\in\Theta_1\cap\Theta_2$ and easily check that
\begin{equation}\label{sum1}
(x_k^*,-y_k^*,-y_k^*)\in\widehat N_{\varepsilon_k}^c\big((x_k,y_{1k},y_{2k});\Theta_1\cap\Theta_2\big)\cap[T(x_k;\Omega)\times Y\times Y]~~\operatorname{for~all}~k\in\mathbb N.
\end{equation}
We know that the sequence $\{(x_k^*,-y_k^*)\}$ is bounded and hence find $L>0$ such that $\|(x_k^*,-y_k^*)\|_{X\times Y}\leq L$ for all $k\in\mathbb N$. Applying the fuzzy rule from Theorem~\ref{fuzzy} to the intersection in \eqref{sum1} along some $\varepsilon_k\downarrow0$ gives us $\lambda_k\ge 0$, $(\tilde x_{ik},\tilde y_{ik})\in\operatorname{gph} S_i|_\Omega$, and 
$$
(x_{ik}^*,y_{ik}^*)\in\big[\widehat N\big((x_k,y_{ik});\operatorname{gph} S_i|_\Omega\big)+(L+1)\varepsilon_k(B_X\times B_Y)\big]\cap[T(x_{ik};\Omega)\times Y]
$$
such that $\|(\tilde x_{ik},\tilde y_{ik})-(x_k,y_{ik})\|_{X\times Y}\leq\varepsilon_k$, $i=1,2$, and
\begin{equation}\label{sum2}
\|\lambda_k(x_k^*,-y_k^*,-y_k^*)-(x_{1k}^*+x_{2k}^*,-y_{1k}^*,-y_{2k}^*)\|_{X\times Y\times Y}\leq\varepsilon_k
\end{equation}
with $\max\{\lambda_k,\|(x_{1k}^*,y_{1k}^*)\|_{X\times Y}\}=1$. We now claim that there exists $\lambda_0>0$ for which $\lambda_k>\lambda_0$ whenever $k\in\mathbb N$. Indeed, assuming the contrary tells us that $\lambda_k\downarrow0$, and then \eqref{sum2} yields $y_{ik}^*\to0$, $i=1,2$, and $x_{1k}^*+x_{2k}^*\to 0$ as $k\to\infty$. Furthermore, we have $\|(x_{ik}^*,-y_{ik}^*)\|_{X\times Y}>\frac{1}{2}$ for all large $k$ and $i=1,2$, which ensures that
$$
(x_{ik}^*,-y_{ik}^*)\in\widehat N_{2(L+1)\varepsilon_k}^c\big((x_k,y_{ik});\operatorname{gph} S_i|_\Omega\big)\cap[T(x_{ik};\Omega)\times Y]
$$
for such $k$ and $i$. 
Since $\{x_{1k}^*\}$ is bounded along a subsequence, there exists $x_0^*\in X$ such that $x_{1k}^*\rightharpoonup^*x_0^*$ and also $x_{2k}^*\rightharpoonup^*-x_0^*$ as $k\to\infty$ without loss of generality. Thus
$$
x_0^*\in D_\Omega^{c*}S_1(\bar x|\bar y_1)(0)\cap(-D_\Omega^{c*}S_2(\bar x|\bar y_2))(0),
$$
which implies by the qualification condition \eqref{sumQC} that $x_0^*=0$. For definiteness, suppose that $S_1$ is PSNC relative to $\Omega$ at $(\bar x,\bar y_1)$. Then $x_{1k}^*\to 0$ that contradicts the condition $\max\{\lambda_k,\|(x_{1k}^*,y_{1k}^*)\|_{X\times Y}\}=1$ due to $\lambda_k\downarrow0$ and $y_{1k}^*\to0$ as $k\to\infty$.

Hence we can put $\lambda_k=1$ without loss of generality. Taking into account that $x_k^*\rightharpoonup^*x^*$ and $y_k^*\to y^*$, it follows from \eqref{sum2} that $y_{ik}^*\to y^*$ and $x_{ik}^*\rightharpoonup^* x_i^*$, $i=1,2$, for some $x_i^*$ with $x^*=x_1^*+x_2^*$. Then we only need to show that $x_i^*\in D_\Omega^{c*}S_i(\bar x|\bar y_i)(y^*)$ for $i=1,2$. To proceed, observe that if $(x_{ik}^*,-y_{ik}^*)\to(0,0)$ as $k\to\infty$, then $(x_i^*,y^*)=(0,0)$, which immediately gives us \eqref{sumrule1}. In the other case, there exists $M>0$ such that $\|(x_{ik}^*,-y_{ik}^*)\|_{X\times Y}>M$ for all $k\in\mathbb N$, and then we get
$$
(x_{ik}^*,-y_{ik}^*)\in\widehat N_{(L+1)\varepsilon_k/M}^c\big((x_k,y_{ik});\operatorname{gph} S_i|_\Omega\big)\cap[T(x_k;\Omega)\times Y]
$$
for all $k\in\mathbb N$ and $i=1,2$. This yields $x_i^*\in D_\Omega^{c*}S_i(\bar x|\bar y_i)(y^*)$ for $i=1,2$, which verifies \eqref{sumrule1} and thus completes the proof of the theorem.
\end{proof}

Finally, the pointbased characterizations of the relative Lipschitz-like property from Section~\ref{sec:point} lead us to the following corollary ensuring the fulfillment of the qualification and relative PSNC conditions of Theorem~\ref{sumrule}.

\begin{corollary}\label{cor:sum} In the general setting of Theorem~\ref{sumrule}, assume that either $S_1$ is relative Lipschitz-like around $(\ox,\oy_1)$, or $S_2$ is relative Lipschitz-like around $(\ox,\oy_2)$. Then both sum rule inclusions \eqref{sumrule1} and \eqref{sumrule2} are satisfied.   \end{corollary}
\begin{proof} Suppose for definiteness that $S_1$ is relative Lipschitz-like around $(\ox,\oy_1)$. Then Theorem~\ref{criterion} tells us that $D^{c*}_\O S_1(\ox|\oy_1)(0)=\{0\}$, and hence the qualification condition \eqref{sumQC} holds. This coderivative criterion also yields the fulfillment of the relative PSNC property from Definition~\ref{psnc}, and thus the claimed statement is verified.
\end{proof}\vspace*{-0.15in}

\section{Concluding Remarks}\label{sec:conc}

This paper introduces and develops new notions of relative generalized normals and corresponding coderivatives for set-valued mappings in Banach spaces and use them for deriving necessary conditions, sufficient conditions, and characterizations of the relative Lipschitz-like property of general multifunctions. Moreover, it is shown that the limiting conic contingent coderivatives in the obtained characterizations enjoy pointbased calculus rules, which justify their broad usage in applications. To achieve such a calculus, we develop a powerful technique of variational analysis resulting in the novel relative extremal principle for systems of sets and the corresponding fuzzy intersection rule.

The major goals of our future research revolve around applications of the newly developed results to various problems of optimization and equilibria. From the technical viewpoint, we plan to relax the convexity assumption on the constraint set in the stability results and calculus rules. Furthermore, we intend to generalize the Hilbert space requirement in the fuzzy intersection and pointbased sum rules derived in Section~\ref{sec:sum}.


\vspace*{-0.15in}

\begin{thebibliography}{30}

\bibitem{Bonnans2000} J. F. Bonnans and A. Shapiro, {\it Perturbation Analysis of Optimization Problems}, Springer, New York, 2000.

\bibitem{Borwein2005} J. M. Borwein and Q.  J. Zhu, {\it Techniques of Variational Analysis}, Springer, New York, 2005.

\bibitem{Browder1983} F. E. Browder, Fixed point theory and nonlinear problems, {\it Bull. Amer. Math. Soc.} {\bf 9} (1983), 1--39.

\bibitem{fabian} M. Fabian et al., {\it Functional Analysis and Infinite-Dimensional Geometry}, 2nd edition, Springer, New York, 2011.

\bibitem{Ioffe} A. D. Ioffe, {\it Variational Analysis of Regular Mappings: Theory and Applications}, Springer, Cham, Switzerland, 2017.

\bibitem{Yang2021} K. W. Meng, M. H. Li, W. F. Yao and X. Q. Yang, Lipschitz-like property relative to a set and the generalized Mordukhovich criterion, {\it Math. Program.} {\bf 189} (2021), 455--489.

\bibitem{Mordukhovich2006} B. S. Mordukhovich, {\it Variational Analysis and Generalized Differentiation, I: Basic Theory, II; Applications}, Springer, Berlin, 2006.

\bibitem{m18}  B. S. Mordukhovich, {\it Variational Analysis and Applications}, Springer, Cham, Switzerland, 2018.

\bibitem{Nam} B. S. Mordukhovich and N. M. Nam, {\it Convex Analysis and Beyond, I: Basic Theory}, Springer, Cham, Switzerland, 2022.

\bibitem{Mordukhovich2023} B. S. Mordukhovich, P. C. Wu and X. Q. Yang, Relative well-posedness of constrained systems with applications to variational inequalities, https://arxiv.org/abs/2212.02727 (2023).

\bibitem{Rockafellar1998} R. T. Rockafellar and R. J-B Wets, {\it Variational Analysis}, Springer, Berlin, 1998.

\bibitem{Thibault} L. Thibault, {\it Unilateral Variational Analysis in Banach Spaces}, published in two volumes, World Scientific, Singapore, 2023.

\bibitem{Yao2023} W. F. Yao, K. W. Meng, M. H. Li and X. Q. Yang, Projectional coderivatives and calculus rules, {\it Set-Valued Var. Anal.} {\bf 31} (2023), 36; https://doi.org/10.1007/s11228-023-00698-9

\end{thebibliography}
\end{document}